\documentclass{amsart}

\usepackage[usenames,dvipsnames]{color}
\usepackage{anysize}
\marginsize{3cm}{3cm}{4cm}{4cm}

\newtheorem{thm}{Theorem}[section]
\newtheorem{cor}[thm]{Corollary}
\newtheorem{lem}[thm]{Lemma}
\newtheorem{prop}[thm]{Proposition}
\theoremstyle{definition}
\newtheorem{defn}[thm]{Definition}
\theoremstyle{remark}
\newtheorem{rem}[thm]{Remark}
\numberwithin{equation}{section}
\newtheorem{example}[thm]{\bf Example}

\newcommand{\set}[1]{\left\{#1\right\}}

\newcommand{\paren}[1]{\left(#1\right)}

\begin{document}

\title[Properties of finite dual fusion frames]{Properties of finite dual fusion frames}

\author[Sigrid B. Heineken]{Sigrid B. Heineken$^{1,*}$}%

\author[Patricia M. Morillas]{Patricia M. Morillas$^2$\\\\ $^{1}$\textit{D\lowercase{epartamento de} M\lowercase{atem\'atica}, FCE\lowercase{y}N, U\lowercase{niversidad de }B\lowercase{uenos} A\lowercase{ires}, P\lowercase{abell\'on} I, C\lowercase{iudad }U\lowercase{niversitaria}, IMAS, UBA-CONICET, C1428EGA C.A.B.A., A\lowercase{rgentina}\\ $^2$ I\lowercase{nstituto de }M\lowercase{atem\'{a}tica }A\lowercase{plicada} S\lowercase{an} L\lowercase{uis, }UNSL-CONICET \lowercase{and} D\lowercase{epartamento de} M\lowercase{atem\'{a}tica}, FCFM\lowercase{y}N, UNSL, E\lowercase{j\'{e}rcito de los }A\lowercase{ndes 950, 5700 }S\lowercase{an} L\lowercase{uis,} A\lowercase{rgentina}}}
%


\thanks{* Corresponding author.\\
\textit{E-mail addresses:}  sheinek@dm.uba.ar (S. B. Heineken),
morillas@unsl.edu.ar (P. M. Morillas)}

\begin{abstract}
A new notion of dual fusion frame has been recently introduced by
the authors. In this article that notion is further motivated and it
is shown that it is suitable to deal with questions posed in a
finite-dimensional real or complex Hilbert space, reinforcing the
idea that this concept of duality solves the question about an
appropriate definition of dual fusion frames. It is shown that for
overcomplete fusion frames there always exist duals different from
the canonical one. Conditions that assure the uniqueness of duals
are given. The relation of dual fusion frame systems with dual
frames and dual projective reconstruction systems is established.
Optimal dual fusion frames for the reconstruction in case of
erasures of subspaces, and optimal dual fusion frame systems for the
reconstruction in case of erasures of local frame vectors are
determined. Examples that illustrate the obtained results are
exhibited.

\bigskip

\bigskip


\medskip


\end{abstract}

\maketitle

\section{Introduction}

A \textit{frame} \cite{Casazza (2000), {Casazza-Kutyniok (2012)},
Christensen (2003), Kovacevic-Chebira (2008)} for a separable
Hilbert space $\mathcal{H}$ is a family of vectors in $\mathcal{H}$
which allow stable and not necessarily unique representations of the
elements of $\mathcal{H}$ via the so-called \emph{dual frames}.
Frames are useful in areas such as signal processing, coding theory,
communication theory and sampling theory, among others.

In  many applications such as distributing sensing, parallel
processing and packet encoding, a distributed processing by
combining locally data vectors has to be implemented. \textit{Fusion
frames} (or \textit{frames of subspaces}) \cite{Casazza-Kutyniok
(2004), Casazza-Kutyniok-Li (2008)} (see also \cite[Chapter
13]{Casazza-Kutyniok (2012)}) are a generalization of frames and
provide a mathematical framework suitable for these applications.
They are collections of closed subspaces and weights, and permit the
reconstruction of each element of $\mathcal{H}$ from packets of
coefficients.

\subsection{Duality in fusion frames.}
Given a frame, the set of dual frames plays a crucial role in
designing suitable reconstruction strategies. In the attempt to
define dual fusion frames appears a technical difficulty related to
the domain of the synthesis operator. A new concept of dual fusion
frame has been proposed by the first author of this paper, which
extends the ``canonical" notion used so far and overcomes this
technical difficulty.

In \cite{Heineken-Morillas-Benavente-Zakowicz (2012)} properties and
examples in infinite-dimensional separable Hilbert spaces are
provided. There the focus is set on questions related to the
boundedness of the operators involved in the definition of duality,
and examples of dual fusion frames are given in $L^2(\mathbb{R}).$

In the present paper we consider instead the finite-dimensional case
studying aspects not addressed in
\cite{Heineken-Morillas-Benavente-Zakowicz (2012)}. In applications,
finite-dimensional Hilbert spaces and finite fusion frames play a
main role \cite{Casazza-Kutyniok (2012)}. They avoid the
approximation problems related to the truncation needed in the
infinite-dimensional case. It is worth to mention that there are
questions which only make sense in the finite-dimensional situation.
This is the case for example for the study of optimal
reconstructions under erasures  (see, e. g., \cite{Casazza-Kutyniok
(2008)}), that is considered in the present paper.

\subsection{Previous approaches.} Other approaches can be considered to study duality of fusion
frames. One of them are the alternate dual fusion frames introduced
in \cite{Gavruta (2007)}. We show that the reconstruction formula
provided by these duals can be obtained using the new concept. One
advantage of the new dual fusion frames with respect to alternate
dual frames is that they can be easily obtained from the left
inverses of the analysis operator of the fusion frames, or from dual
frames.

Fusion frames can be viewed as a particular case of g-frames
\cite{Sun (2006)}, so one attempt could be to study duality of
fusion frames in the context of dual g-frames. For example, in
\cite{Arefijamaal-Ghasemi (2013)} dual g-frames with respect to the
same family of subspaces are considered, but this would have no
sense applied to the study of duality of fusion frames.
Reconstruction systems are g-frames in finite-dimensional Hilbert
spaces. In this setting, duality of fusion frames was studied
viewing them as projective reconstruction systems
\cite{Massey-Ruiz-Stojanoff (2012a), Massey-Ruiz-Stojanoff (2012b)}.
But projective reconstruction systems are not closed under duality,
more precisely, there exist projective reconstruction systems with
non projective canonical dual or without any projective dual
\cite{Massey-Ruiz-Stojanoff (2012b)}. This drawback also appears in
the setting considered in \cite{Krishtal (2011)}. We note that these
problems are not present if we use the new definition of dual fusion
frames.

\subsection{Optimal reconstruction under erasures.} In real implementations often some of the data vectors, or part of
them, are lost or erased, and it is necessary to perform the
reconstruction with the partial information at hand.

One approach to address this situation is to derive sufficient
conditions for a fusion frame to be robust to such erasures, and
construct fusion frames that are optimally robust. Here robustness
is understood as a certain minimizing reconstruction error property.
This approach is considered, e. g., in \cite{Casazza-Kutyniok
(2008)} for tight fusion frames using the canonical dual for the
reconstruction.

In applications there might be several restrictions when selecting
fusion frames for encoding, that make it impossible to find one that
is optimally robust. The new concept of dual fusion frames allows
another approach, studying how to select optimal dual fusion frames
for a fixed fusion frame. In particular, in this article we analyze
this question when a blind reconstruction process is used, in a
similar way as it was done in \cite{Lopez-Han (2010), Leng-Han
(2011)} for frames and in \cite{Massey-Ruiz-Stojanoff (2012b)} for
projective reconstruction systems. As in these works,
 we obtain, under certain conditions, a unique optimal dual fusion frame of
a given fusion frame. We note that in \cite{Massey-Ruiz-Stojanoff
(2012b)}, it is shown that the optimal dual reconstruction system is
not necessarily projective, so it can not always be viewed as a
fusion frame.

\subsection{Contents.} In Section 2, we briefly
review frames, fusion frames and fusion frame systems.

In Section 3  we present the new concept of dual fusion frame. Then
we consider two special cases: block-diagonal and component
preserving duals, for which the reconstruction formula has a simpler
expression. We present a characterization of component preserving
dual fusion frames in terms of the left inverses of the analysis
operator of the original fusion frame. We then refer to the duals
defined in \cite{Casazza-Kutyniok (2004)}. These duals are component
preserving and we call them canonical. We prove that for
overcomplete fusion frames with non trivial subspaces, there always
exist component preserving dual fusion frames different from the
canonical ones. The new definition of dual fusion frames is a
generalization of conventional dual frames and it provides more
flexibility. For instance, a Riesz fusion basis can have only one
component preserving dual but we show that it can have more than one
non component dual, unless additional conditions are imposed.

In Section 4,  we introduce a linear transformation that links the
analysis operator of a fusion frame system with the analysis
operator of its associated frame. Using this transformation, we
define dual fusion frame systems, which are block-diagonal. We
establish the close relation of dual fusion frame systems with dual
frames and dual projective reconstruction systems, showing that the
new definition of dual fusion frames arises naturally.

In Section~\ref{S ODFFE}, we determine the duals that minimize the
mean square error and the worst case error in the presence of
erasures when a blind reconstruction process is used. In both cases,
we determine optimal dual fusion frames for the reconstruction in
case of erasures of subspaces and optimal dual fusion frame systems
for the reconstruction in case of erasures of local frame vectors.

Finally, in Section 6, we show that the reconstruction formula
provided by the alternate dual fusion frames introduced in
\cite{Gavruta (2007)} can be obtained using the new concept of dual
fusion frame. We also present examples that illustrate the results
described before.


\section{Preliminaries}

In this section we review the concepts of frame \cite{Casazza
(2000), Casazza-Kutyniok (2012), Christensen (2003),
Kovacevic-Chebira (2008)}, fusion frame and fusion frame system
\cite{Casazza-Kutyniok (2004), Casazza-Kutyniok-Li (2008)} (see also
\cite[Chapter 13]{Casazza-Kutyniok (2012)}). We refer to the
mentioned works for more details. We begin introducing some
notation.

\subsection{Notation}

Let $\mathcal{H}, \mathcal{K}$ be finite-dimensional Hilbert spaces
over $\mathbb{F}=\mathbb{R}$ or $\mathbb{F}=\mathbb{C}$. Let
$L(\mathcal{H},\mathcal{K})$ be the space of linear transformations
from $\mathcal{H}$ to $\mathcal{K}$ (we write $L(\mathcal{H})$ for
$L(\mathcal{H},\mathcal{H})$). Given $T \in
L(\mathcal{H},\mathcal{K})$ we write $R(T)$, $N(T)$ and $T^{*}$ to
denote the image, the null space and the adjoint of $T$,
respectively. If $T \in L(\mathcal{H},\mathcal{K})$ is injective,
$\mathfrak{L}_{T}$ denotes the set of left inverses of $T$.

The inner product and the norm in $\mathcal{H}$ will be denoted by
$\langle.,.\rangle_{\mathcal{H}}$ and $\|.\|_{\mathcal{H}}$,
respectively. If $T \in L(\mathcal{H},\mathcal{K})$, then
$\|T\|_{F}$ and $\|T\|_{sp}$ denote the Frobenius and the spectral
norms of $T$, respectively.

If $V \subset \mathcal{H}$ is a subspace, $\pi_{V} \in
L(\mathcal{H})$ denotes the orthogonal projection onto $V$.

Let $m,n,d \in \mathbb{N}$ and $\mathbf{n}=(n_{1},\dots,n_{m}) \in
\mathbb{N}^{m}$. In the sequel, $\mathcal{H}$ will be a
finite-dimensional Hilbert space over $\mathbb{F}$ of dimension $d$.
For $J \subseteq \{1, \ldots, m\}$ let $\chi_{J} : \{1, \ldots, m\}
\rightarrow \{0, 1\}$ be the characteristic function of $J.$ We
abbreviate $\chi_{\{j\}}=\chi_{j}$. For
$p\in\mathbb{N}\cup\{\infty\}$ let $\|.\|_{p}$ denote the $p$-norm
in $\mathbb{F}^{n}$.

\subsection{Frames}

\begin{defn}\label{D F}
Let $\mathcal{F}=\{f_{i}\}_{i=1}^{m} \subset \mathcal{H}$.
\begin{enumerate}
  \item The \textit{synthesis operator} of $\mathcal{F}$ is

\centerline{ $T_{\mathcal{F}} : \mathbb{F}^{m} \rightarrow
\mathcal{H},$
$T_{\mathcal{F}}(x_{i})_{i=1}^{m}=\sum_{i=1}^{m}x_{i}f_{i}$}
\noindent and the \textit{analysis operator} is

\centerline{$T_{\mathcal{F}}^{*} : \mathcal{H} \rightarrow
\mathbb{F}^{m},$ $ T_{\mathcal{F}}^{*}f=(\langle
f,f_{i}\rangle)_{i=1}^{m}.$}
  \item $\mathcal{F}$ is a \textit{frame} for $\mathcal{H}$ if
$\text{span}~\mathcal{F}=\mathcal{H}$.
  \item If $\mathcal{F}$ is a frame for $\mathcal{H},$

\centerline{ $S_{\mathcal{F}}=T_{\mathcal{F}}T_{\mathcal{F}}^{*},$
$S_{\mathcal{F}}f=\sum_{i=1}^{m}\langle f,f_{i}\rangle f_{i},$}
\noindent is the \textit{frame operator} of $\mathcal{F}$.
\end{enumerate}
\end{defn}

The set $\mathcal{F}=\{f_{i}\}_{i=1}^{m} \subset \mathcal{H}$ is a
frame for $\mathcal{H}$ if and only if there exist $\alpha, \beta >
0$ such that
\begin{equation}\label{E cond f}
\alpha\|f\|^{2} \leq \sum_{i=1}^{m}|\langle f,f_{i}\rangle |^{2}
\leq \beta\|f\|^{2} \text{ for all $f \in \mathcal{H}$}.
\end{equation}
We call $\alpha$ and $\beta$ the \textit{frame bounds}. The
\textit{optimal lower frame bound} is
$\|S_{\mathcal{F}}^{-1}\|^{-1}$ and the \textit{optimal upper frame
bound} is $\|S_{\mathcal{F}}\|=\|T_{\mathcal{F}}\|^2.$ The set
$\mathcal{F}$ is an $\alpha$-\textit{tight frame}, if in (\ref{E
cond f}) the constants $\alpha$ and $\beta$ can be chosen so that
$\alpha = \beta$, or equivalently, $S_{\mathcal{F}}=\alpha
I_{\mathcal{H}}$. If $\alpha = \beta = 1$, $\mathcal{F}$ is a
\textit{Parseval frame}.

In frame theory each $f \in \mathcal{H}$ is represented by the
collection of scalar coefficients $\langle f,f_{i}\rangle$, $i = 1,
\ldots, m$, that can be thought as a measure of the projection of
$f$ onto each frame vector. From these coefficients $f$ can be
recovered using a reconstruction formula via the so-called dual
frames.
\begin{defn}\label{D frame dual}
Let $\mathcal{F}=\{f_{i}\}_{i=1}^{m}$ and
$\widetilde{\mathcal{F}}=\{\tilde{f}_{i}\}_{i=1}^{m}$ be frames for
$\mathcal{H}$. Then $\widetilde{\mathcal{F}}$ is a \textit{dual
frame} of $\mathcal{F}$ if the following reconstruction formula
holds \begin{equation}\label{E vectorial D frame dual}
f=\sum_{i=1}^{m}\langle f,f_{i}\rangle \widetilde{f}_{i}, \text{ for
all $f \in \mathcal{H}$},\end{equation}or equivalently,
\begin{equation}\label{E operadores D
frame dual}
T_{\widetilde{\mathcal{F}}}T_{\mathcal{F}}^{*}=I_{\mathcal{H}}.\end{equation}
\end{defn}
Let $\mathcal{F}=\{f_{i}\}_{i=1}^{m}$ be a frame for $\mathcal{H}$.
Then $\{S_{\mathcal{F}}^{-1}f_{i}\}_{i=1}^{m}$ is the
\textit{canonical dual frame} of $\mathcal{F}$.

\subsection{Fusion frames and fusion frame systems}

Fusion frames generalize the concept of frames. The representation
of each $f \in \mathcal{H}$ via fusion frames is given by
projections onto multidimensional subspaces, which also satisfy some
stability conditions.
\begin{defn}
Let $\{W_i\}_{i=1}^{m}$ be a family of subspaces of $\mathcal{H}$,
and let $\{w_{i}\}_{i=1}^{m}$ be a family of weights, i.e., $w_{i}
> 0$ for $i=1,\ldots,m$. Then $\{(W_i,w_{i}\}_{i=1}^{m}$ is
called a \textit{Bessel fusion sequence} for $\mathcal{H}$.
\end{defn}

We will denote $\{W_{i}\}_{i=1}^{m}$ with ${\bf W}$,
$\{w_{i}\}_{i=1}^{m}$ with ${\bf w}$ and $\{(W_i,w_{i})\}_{i=1}^{m}$
with $(\textbf{W},{\bf w})$. If $T \in L(\mathcal{H},\mathcal{K})$
we will write $(T\textbf{W},{\bf w})$ for $\{(TW_i,
w_{i})\}_{i=1}^{m}.$

Let $\mathcal{W}:=\bigoplus_{i=1}^{m} W_i = \{(f_{i})_{i=1}^{m}:
f_{i} \in W_i\}$ be the Hilbert space with
$\langle(f_{i})_{i=1}^{m},(g_{i})_{i=1}^{m}\rangle_{
\mathcal{W}}=\sum_{i=1}^{m}\langle f_{i}, g_{i}\rangle.$

\begin{defn}
Let $(\textbf{W},{\bf w})$ be a Bessel fusion sequence.
\begin{enumerate}
  \item $(\textbf{W},{\bf w})$ is called
$w$-\textit{uniform}, if $w_{i}=w$ for all $i \in \{1, \ldots, m\}$.
In this case we write $(\textbf{W},w)$.
  \item $(\textbf{W},{\bf w})$ is called
$n$-\textit{equi-dimensional}, if $\dim(W_i)=n$ for all $i \in \{1,
\ldots, m\}$.
  \item  The
\textit{synthesis operator} of $(\textbf{W},{\bf w})$ is

\centerline{ $T_{\textbf{W},{\bf w}} : \mathcal{W} \rightarrow
\mathcal{H}$, $T_{\textbf{W},{\bf
w}}(f_i)_{i=1}^{m}=\sum_{i=1}^{m}w_{i}f_{i}.$}

\noindent The \textit{analysis operator} is

\centerline{$T_{\textbf{W},{\bf w}}^{*} : \mathcal{H} \rightarrow
\mathcal{W}$, $T_{\textbf{W},{\bf
w}}^{*}f=(w_{i}\pi_{W_i}(f))_{i=1}^{m}.$}
  \item $(\textbf{W},{\bf w})$ is called a
  \textit{fusion frame} for $\mathcal{H}$ if
$\text{span}\bigcup_{i=1}^{m}W_i=\mathcal{H}$.
  \item $(\textbf{W},{\bf w})$ is a \textit{Riesz
fusion basis} if $\mathcal{H}$ is the direct sum of the $W_i$, and
$(\textbf{W},1)$ an \textit{orthonormal fusion basis} if
$\mathcal{H}$ is the orthogonal sum of the $W_i.$
  \item If $(\textbf{W},{\bf w})$  is a fusion frame for $\mathcal{H}$, the
operator

\centerline{
$S_{\textbf{W},\mathbf{w}}=T_{\textbf{W},\mathbf{w}}T_{\textbf{W},\mathbf{w}}^{*}
: \mathcal{H} \rightarrow \mathcal{H}\text{ ,
}S_{\textbf{W},\mathbf{w}}(f)=T_{\textbf{W},\mathbf{w}}T_{\textbf{W},\mathbf{w}}^{*}(f)=\sum_{i=1}^{m}w_{i}^{2}\pi_{W_i}(f)$}

\noindent is called the \textit{fusion frame operator} of
$(\textbf{W},\mathbf{w})$.
\end{enumerate}
\end{defn}

A Bessel fusion sequence $(\textbf{W},{\bf w})$ is a fusion frame
for $\mathcal{H}$ if and only if $T_{\textbf{W},{\bf w}}$ is onto,
or equivalently, if and only if there exist constants $0 < \alpha
\leq \beta < \infty$ such that
\begin{equation}\label{E cond ff}
\alpha\|f\|^{2} \leq \sum_{i=1}^{m}w_{i}^{2}\|\pi_{W_i}(f)\|^{2}
\leq \beta\|f\|^{2}  \text{ for all $f \in \mathcal{H}$.}
\end{equation}
We call $\alpha$ and $\beta$ the \textit{fusion frame bounds}. A
fusion frame $(\textbf{W},{\bf w})$ is called an
$\alpha$-\textit{tight fusion frame} if in (\ref{E cond ff}) the
constants $\alpha$ and $\beta$ can be chosen so that $\alpha =
\beta$, or equivalently, $S_{\textbf{W},{\bf w}}=\alpha
I_{\mathcal{H}}$. If $\alpha = \beta = 1$ we say that it is a
\textit{Parseval fusion frame}.

The use of fusion frames permits furthermore local processing in
each of the subspaces. For this, it is useful to have a set of local
frames for its subspaces:

\begin{defn}
Let $(\textbf{W},{\bf w})$ be a fusion frame for $\mathcal{H}$, and
let $\{f_i^l\}_{l\in L_i}$ be a frame for $W_{i}$ for
$i=1,\ldots,m$. Then we call $\{(W_i,w_i,\{f_i^l\}_{l\in
L_i})\}_{i=1}^{m}$ a \emph{fusion frame system} for $\mathcal{H}$.
\end{defn}

From now on we denote $\mathcal{F}_{i}=\{f_i^l\}_{l \in L_i}$,
$\mathcal{F}=\{\mathcal{F}_{i}\}_{i=1}^{m},$ ${\bf
w}\mathcal{F}=\{w_{i}\mathcal{F}_{i}\}_{i=1}^{m},$ and we will
abbreviate $\{(W_i,w_i,\{f_i^l\}_{l\in L_i})\}_{i=1}^{m}$ with
$(\textbf{W}, {\bf w}, \mathcal{F}).$  If $T \in
L(\mathcal{H},\mathcal{K})$ we use the notation $T\mathcal{F}$ for
$\{\{T{f}_i^l\}_{l\in L_i}\}_{i=1}^m.$

\begin{rem}\label{R wF marco sii WwF fusion frame system} Clearly,
$\mathbf{w}\mathcal{F}$ is a frame for $\mathcal{H}$ if and only if
$(\mathcal{W},\mathbf{w},\mathcal{F})$ is a fusion frame system for
$\mathcal{H}$.
\end{rem}

\section{Dual fusion frames}

One of the most important properties of frames is that they permit
different representations for each element of $\mathcal{H},$ which
are provided by the duals via the reconstruction formula (\ref{E
vectorial D frame dual}). Taking this into account, our purpose is
to have a notion of dual fusion frame as we have it in the classical
frame theory, and that furthermore leads to analogous results. We
note that for frames the duality condition can be expressed in two
forms: (\ref{E vectorial D frame dual}) and (\ref{E operadores D
frame dual}). So, it is natural to try to generalize these
expressions to the context of fusion frames in order to obtain a
definition of dual fusion frame.

Let $(\textbf{W},{\bf w})$ be a fusion frame. Since
$S^{-1}_{\textbf{W},{\bf w}}S_{\textbf{W},{\bf w}}=I_{\mathcal{H}}$,
we have the following reconstruction formula
\begin{equation}\label{E reconst dual fusion frame canonico}
f=\sum_{i=1}^{m}w_{i}^{2}S_{\textbf{W},{\bf w}}^{-1}\pi_{W_{i}}(f)
\text{, for all } f \in \mathcal{H},
\end{equation} \noindent that is analogous to (\ref{E vectorial D frame dual}). The family $(S_{\textbf{W},{\bf w}}^{-1}\mathbf{W},\mathbf{w})$ is a
fusion frame which in \cite[Definition 3.19]{Casazza-Kutyniok
(2004)} is called the dual fusion frame of $(\textbf{W},{\bf w}),$
and is similar to the canonical dual frame in the classical frame
theory. As it is pointed out in  \cite[Chapter 13]{Casazza-Kutyniok
(2012)}, (\ref{E reconst dual fusion frame canonico}) - in contrast
to the analogous one for frames - does not lead automatically to a
dual fusion frame concept.

So, instead of trying to generalize (\ref{E vectorial D frame
dual}), we can try with (\ref{E operadores D frame dual}). But in
this case we find the following obstacle. Given $(\textbf{W},{\bf
w})$ and $({\bf V},{\bf v})$  two fusion frames for $\mathcal{H}$,
with ${\bf W} \neq {\bf V}$, the corresponding synthesis operators
$T_{\textbf{W},{\bf w}}$ and $T_{{\bf V},{\bf v}}$ have different
domains. Therefore the composition of $T_{{\bf V},{\bf v}}$ with
$T^{*}_{\textbf{W},{\bf w}}$ is not possible. The next definition
overcomes this problem, extends the notion introduced in
\cite{Casazza-Kutyniok (2004)} (see subsection~\ref{Ej dual
canonico}) and, as we are going to see, leads to the properties that
we would desire a dual fusion frame to have.

\begin{defn}\label{D fusion frame dual}
Let $(\textbf{W},{\bf w})$ and $({\bf V},{\bf v})$ be two fusion
frames for $\mathcal{H}$. Then $({\bf V},{\bf v})$ is a dual fusion
frame of $(\textbf{W},{\bf w})$ if there exists a $Q \in
L(\mathcal{W}, \mathcal{V})$ such that
\begin{equation}\label{E TvQTw*=I}
T_{{\bf V},{\bf v}}QT^{*}_{\textbf{W},{\bf w}}=I_{\mathcal{H}}.
\end{equation}
\end{defn}

If we need to do an explicit reference to the linear transformation
$Q$ we say that $({\bf V},{\bf v})$ is a $Q$-dual fusion frame of
$(\textbf{W},{\bf w})$.  Note that if $({\bf V},{\bf v})$ is a
$Q$-dual fusion frame of $(\textbf{W},{\bf w})$, then $({\bf W},{\bf
w})$ is a dual $Q^{*}$-dual fusion frame of $(\textbf{V},{\bf v}).$

Note that in (\ref{E operadores D frame dual}) the operator
``between" $T_{\widetilde{\mathcal{F}}}$ and $T_{\mathcal{F}}^{*}$
is $I_{\mathbb{F}^m},$ which is hidden. In view of this, (\ref{E
TvQTw*=I}) can be seen as a generalization of (\ref{E operadores D
frame dual}).

Now we introduce two particular types of linear transformations $Q$
for which the reconstruction formula obtained from (\ref{E
TvQTw*=I}) is simpler. For this, we consider the selfadjoint
operator $M_{J,\textbf{W}} : \mathcal{W} \rightarrow \mathcal{W},
M_{J,\textbf{W}}(f_j)_{j=1}^{m}=(\chi_{J}(j)f_{j})_{j=1}^{m}.$ We
simply write $M_{J}$ if it clear to which $\textbf{W}$ we refer to.
We abbreviate $M_{\{j\},\textbf{W}}=M_{j,\textbf{W}}$ and
$M_{\{j\}}=M_{j}$.

\begin{defn}Let $Q \in L(\mathcal{W}, \mathcal{V})$.
\begin{enumerate}
\item If $QM_{j,\textbf{W}}\mathcal{W}
\subseteq M_{j, {\bf V}}\mathcal{V}$ for each $j \in \{1, \ldots,
m\},$ $Q$ is called \emph{block-diagonal}.

\item If $QM_{j,\textbf{W}}\mathcal{W}
= M_{j, {\bf V}}\mathcal{V}$ for each $j \in \{1, \ldots, m\},$ $Q$
is called {\em component preserving}.
\end{enumerate}
\end{defn}

Observe that $Q$ is block-diagonal if and only if
$QM_{J,\textbf{W}}=M_{J,\mathbf{V}}Q$ for each $J \subseteq \{1,
\ldots, m\}$, or equivalently, $QM_{j,\textbf{W}}=M_{j, {\bf V}}Q$
for each $j \in \{1, \ldots, m\}$. If $Q$ is block-diagonal, then
$Q^{*}$ is block-diagonal. If in Definition~\ref{D fusion frame
dual} $Q$ is block-diagonal (component preserving) we say that
$({\bf V},{\bf v})$ is a \emph{block-diagonal dual fusion frame}
(\emph{component preserving dual fusion frame}) of $(\textbf{W},{\bf
w})$.

It is important to note, as we will see in Theorem~\ref{T V,v dual
fusion frame sii Vi=ApiWj}, that $Q$ is component preserving for
dual fusion frames obtained from the left inverses of
$T^{*}_{\textbf{W},{\bf w}}.$ Also, $Q$ is block-diagonal for dual
fusion frame systems (see Definition~\ref{D fusion frame system
dual} and Remark~\ref{R Q Mi sum Wj subset Mi sum Vj}).

The reconstruction formula following from (\ref{E TvQTw*=I}) has the
form \begin{equation}\label{E fc} f
=\sum_{j=1}^{m}v_{j}(Q(w_{i}\pi_{W_{i}}f)_{i=1}^{m})_{j}~,~~\forall
f \in \mathcal{H}.
\end{equation}
If $Q$ is block-diagonal (\ref{E fc}) becomes
\begin{equation}\label{E fcs} f
=\sum_{j=1}^{m}v_{j}w_{j}Q_{j}f~,~~\forall f \in
\mathcal{H},\end{equation} where
$Q_{j}f:=(QM_{j}(\pi_{W_{i}}f)_{i=1}^{m})_{j} \in V_{j}$.

The main advantage of (\ref{E fcs}) over (\ref{E fc}) is that in
(\ref{E fcs}) the $j$-th term is obtained using only the projection
onto $W_j$, whereas in (\ref{E fc}) all the projections onto $W_i$
for $i=1,\ldots, m$ are involved. This fact is particulary useful
for truncation purposes. In this case we can consider an index
subset $J$ in (\ref{E fcs}), to obtain an approximate reconstruction
formula where only the subspaces $W_j$ (and $V_j$) for $j \in J$ are
used.

The linear transformation $Q$ has in many cases a very simple
expression and consequently the reconstruction formula is very
simple too (see examples in subsection~\ref{Ej dual canonico} and
Section~\ref{SEjemplos}).

Now we are going to relate the duals of a fusion frame with the left
inverses of its analysis operator, in a similar fashion as for
frames (see, e. g., \cite[Lemma 5.6.3.]{Christensen (2003)}). For
this, given $A \in L(\mathcal{W}, \mathcal{H})$ and $\bf v$ a
collection of weights, we consider the subspaces
$V_{i}=AM_{i}\mathcal{W}$, for each $i=1,\ldots,m,$ and the linear
transformation $$Q_{A,{\bf v}}: \mathcal{W} \rightarrow \mathcal{V}
\text{ , }Q_{A,{\bf
v}}(f_j)_{j=1}^{m}=\paren{\frac{1}{v_{i}}AM_{i}(f_j)_{j=1}^{m}}_{i=1}^{m}.$$
Its adjoint is $Q_{A,{\bf v}}^{*}: \mathcal{V} \rightarrow
\mathcal{W},$ $Q_{A,{\bf
v}}^{*}(g_{j})_{j=1}^{m}=\sum_{i=1}^{m}\frac{1}{v_{i}}M_{i}A^{*}g_{i}$.

To simplify the exposition, we just formulate the next lemmas which
proofs are straightforward.

\begin{lem}\label{L QA comp preserv}
Let $(\textbf{W},{\bf w})$ be a Bessel fusion sequence for
$\mathcal{H}$, $A \in L(\mathcal{W}, \mathcal{H})$, $\bf v$ a
collection of weights and $V_{i}=AM_{i}\mathcal{W}$, for each $i \in
\{1,\ldots,m\}$. Then $Q_{A,{\bf v}}$ is component preserving and
$A=T_{{\bf V},{\bf v}}Q_{A,{\bf v}}$.
\end{lem}

\begin{lem}\label{L W V Vtild Q Qtild comp preserv}
Let $(\textbf{W},{\bf w})$, $({\bf V},{\bf v})$ and
$({\bf\widetilde{V}},{\bf \widetilde{v}})$ be Bessel fusion
sequences for $\mathcal{H}$. Let $Q \in L(\mathcal{W}, \mathcal{V})$
and $\widetilde{Q} \in L(\mathcal{W}, \widetilde{\mathcal{V}})$. If
$T_{{\bf V},{\bf v}}Q=T_{\widetilde{V},\widetilde{v}}\widetilde{Q},$
then the following assertions hold:
\begin{enumerate}
  \item If $Q$ and $\widetilde{Q}$ are component preserving then $V_{i}=\widetilde{V}_{i}$ for $i = 1,\ldots, m$.
  \item Let $D : \mathcal{V} \rightarrow \mathcal{V},$
$D(g_{i})_{i=1}^{m}=(\frac{\widetilde{v}_{i}}{v_{i}}g_{i})_{i=1}^{m}$.
If $Q$ and $\widetilde{Q}$ are block-diagonal and
$\mathbf{V}=\widetilde{\mathbf{V}}$, then $Q=D\widetilde{Q}$.
\end{enumerate}
\end{lem}

The next theorem characterizes the component preserving dual fusion
frames of $(\textbf{W},{\bf w})$ in terms of the left inverses of
$T^{*}_{\textbf{W},{\bf w}}$:

\begin{thm}\label{T V,v dual fusion frame sii Vi=ApiWj}
Let $(\textbf{W},{\bf w})$ be a fusion frame for $\mathcal{H}$. Then
$({\bf V},{\bf v})$ is a $Q$-component preserving dual fusion frame
of $(\textbf{W},{\bf w})$ if and only if $V_{i}=AM_{i}\mathcal{W}$
for each $i\in\{1,\ldots,m\}$ and $Q=Q_{A,{\bf v}},$ for some $A \in
\mathfrak{L}_{T^{*}_{\textbf{W},{\bf w}}}$. Moreover, any element of
$\mathfrak{L}_{T^{*}_{\textbf{W},{\bf w}}}$ is of the form $T_{{\bf
V},{\bf v}}Q$ where $({\bf V},{\bf v})$ is some $Q$-component
preserving dual fusion frame of $(\textbf{W},{\bf w})$.
\end{thm}

\begin{proof}
Let $(\mathbf{V},{\bf v})$ be a $Q$-component preserving dual fusion
frame of $(\mathbf{W},{\bf w})$. Since $T_{\mathbf{V},{\bf
v}}QT^{*}_{\mathbf{W},{\bf w}}=I_{\mathcal{H}}$,
$A:=T_{\mathbf{V},{\bf v}}Q \in\mathfrak{L}_{T^{*}_{\textbf{W},{\bf
w}}}$. Since $Q$ is component preserving,
$AM_{i}\mathcal{W}=T_{\mathbf{V},{\bf v}}QM_{i}\mathcal{W}=V_{i}$.
By Lemma~\ref{L QA comp preserv}, $Q_{A,{\bf v}}$ is component
preserving and $T_{\mathbf{V},{\bf v}}Q=A=T_{\mathbf{V},{\bf
v}}Q_{A,{\bf v}}$. By Lemma~\ref{L W V Vtild Q Qtild comp preserv}
(2), $Q=Q_{A,{\bf v}}$.

Let now $A\in\mathfrak{L}_{T^{*}_{\textbf{W},{\bf w}}}$ and
$V_{i}=AM_{i}\mathcal{W}$ for each $i\in\set{1,...,m}.$ By
Lemma~\ref{L QA comp preserv}, $Q_{A,{\bf v}}$ is component
preserving and $A=T_{\mathbf{V},{\bf v}}Q_{A,{\bf v}}$. Since
$A\in\mathfrak{L}_{T^{*}_{\textbf{W},{\bf w}}}$, $T_{\mathbf{V},{\bf
v}}Q_{A,{\bf v}}T^{*}_{\mathbf{W},{\bf w}}=I_{\mathcal{H}}$. So
$(\mathbf{V},{\bf v})$ is a $Q_{A,{\bf v}}$-component preserving
dual fusion frame of $(\mathbf{W},{\bf w})$.

The last assertion of the theorem follows from the previous steps of
the proof.
\end{proof}

\begin{rem}
By Theorem~\ref{T V,v dual fusion frame sii Vi=ApiWj}, we can always
associate to any $Q$-dual fusion frame $({\bf V},{\bf v})$ of
$(\textbf{W},{\bf w})$ the $Q_{A,{\bf \tilde{v}}}$-component
preserving dual fusion frame
$\{(AM_{i}\mathcal{W},\tilde{v}_{i})\}_{i=1}^{m}$ where $A=T_{{\bf
V},{\bf v}}Q$ and $\{\widetilde{v}_{i}\}_{i=1}^{m}$ are arbitrary
weights. Moreover if $Q$ is block-diagonal, then $Q_{T_{{\bf V},{\bf
v}}Q,{\bf v}}(f_j)_{j=1}^{m}=Q(f_j)_{j=1}^{m}$ for each
$(f_j)_{j=1}^{m} \in \mathcal{W}$.
\end{rem}

\subsection{The canonical dual fusion frame} \label{Ej dual canonico}

If $(\textbf{W},{\bf w})$ is a fusion frame, then
$(S_{\textbf{W},{\bf w}}^{-1}\textbf{W},{\bf w})$ is the dual fusion
frame of $(\textbf{W},{\bf w})$ in the sense of
\cite{Casazza-Kutyniok (2004)}.

By Theorem~\ref{T V,v dual fusion frame sii Vi=ApiWj},
$(S_{\textbf{W},{\bf w}}^{-1}\textbf{W},{\bf v})$ is a $Q_{A,{\bf
v}}$-component preserving dual with $A=S_{\textbf{W},{\bf
w}}^{-1}T_{\textbf{W},{\bf w}}$, ${\bf v}$ a family of arbitrary
weights and $Q_{A,{\bf v}} : \mathcal{W} \rightarrow
\bigoplus_{i=1}^{m} S_{\textbf{W},{\bf w}}^{-1}W_i$, $Q_{A,{\bf
v}}(f_j)_{j=1}^{m}=(\frac{w_{i}}{v_{i}}S_{\textbf{W},{\bf
w}}^{-1}f_{i})_{i=1}^{m}.$  In the sequel we refer to this
$Q_{S_{\textbf{W},{\bf w}}^{-1}T_{\textbf{W},{\bf w}},{\bf v}}$-dual
fusion frame as the {\em canonical dual with weights ${\bf v}$}.

Note that with a canonical dual fusion frame we have the
reconstruction formula (\ref{E reconst dual fusion frame canonico})
that can be written as
\begin{equation}\label{E TvQTw*=I canonico}
f=T_{S_{\textbf{W},{\bf w}}^{-1}W,{\bf v}}Q_{S_{\textbf{W},{\bf
w}}^{-1}T_{\textbf{W},{\bf w}},{\bf v}}T_{W,{\bf w}}^{*}f,~~~\forall
f \in \mathcal{H},
\end{equation} whereas with another dual fusion frame, (\ref{E TvQTw*=I}) provides other alternatives for the reconstruction.

\subsection{Existence of non-canonical dual fusion frames}\label{Existence of non-canonical}

A Bessel fusion sequence $(\textbf{W},{\bf w})$ is a Riesz fusion
basis if and only if $T_{\textbf{W},{\bf w}}^{*}$ is bijective, or
equivalently, $T_{\textbf{W},{\bf w}}^{*}$ has a unique left
inverse. So, if $(\textbf{W},{\bf w})$ is a Riesz fusion basis, the
only component preserving duals of $(\textbf{W},{\bf w})$ are the
canonical ones, and if we fix the weights ${\bf v}$, by Lemma~\ref{L
W V Vtild Q Qtild comp preserv}(2), the unique $Q$ for this dual is
$Q_{S_{\textbf{W},{\bf w}}^{-1}T_{\textbf{W},{\bf w}},{\bf v}}$. A
Riesz fusion basis can have other dual fusion frames (see
Example~\ref{Ej MasseyRuizStojanoff 2}). We will give now conditions
that guarantee the uniqueness of the subspaces of the duals of a
Riesz fusion basis.

\begin{prop}\label{DualRFB}
Let $(\textbf{W},{\bf w})$ be a Riesz fusion basis and ${\bf v}$ a
family of weights. The following assertions hold:
\begin{enumerate}

  \item \label{MC} Let $({\bf V},{\bf v})$ be a block-diagonal dual fusion frame of  $(\textbf{W},{\bf w})$. Then, for each $i = 1, \ldots, m$, $S_{\textbf{W},{\bf w}}^{-1}W_{i} \subseteq
  V_{i}$.
\item \label{CFRB} If $({\bf V},{\bf v})$ is a Riesz fusion basis which is a block-diagonal dual fusion frame of $(\textbf{W},{\bf w}),$ then $V_i=S_{\textbf{W},{\bf w}}^{-1}W_{i}$ for $i=1,\ldots, m.$

     \end{enumerate}
\end{prop}

\begin{proof}(\ref{MC}) Let $f_{i} \in W_{i}$. Using that $T_{{\bf V},{\bf
v}}QT_{\textbf{W},{\bf w}}^{*}=I_{\mathcal{H}}$ and that
$T_{\textbf{W},{\bf w}}^{*}S_{\textbf{W},{\bf
w}}^{-1}T_{\textbf{W},{\bf w}}=I_{\mathcal{W}}$ (the last equation
holds since $(\textbf{W},{\bf w})$ is a Riesz fusion basis), we
obtain

\centerline{ $S_{\textbf{W},{\bf w}}^{-1}f_{i}=T_{{\bf V},{\bf
v}}QT_{\textbf{W},{\bf w}}^{*}S_{\textbf{W},{\bf
w}}^{-1}T_{\textbf{W},{\bf w}}M_{i}(\frac{1}{w_{i}}f_{j})_{j=1}^{m}=
T_{{\bf V},{\bf v}}QM_{i}(\frac{1}{w_{i}}f_{j})_{j=1}^{m}=T_{{\bf
V},{\bf v}}M_{i}Q(\frac{1}{w_{i}}f_{j})_{j=1}^{m}\in V_{i}.$}

\noindent So $S_{\textbf{W},{\bf w}}^{-1}W_{i} \subseteq V_{i}.$

\medskip

(\ref{CFRB}) By (\ref{MC}), $S_{\textbf{W},{\bf
w}}^{-1}W_{i}\subseteq V_i$  for $i=1,\ldots,m.$ If there exists
$i_0\in \{1,\dots,m\}$ such that $S_{\textbf{W},{\bf
w}}^{-1}W_{i_0}\subset V_{ i_0},$ then
$\sum_{i=1}^{m}\dim(S_{\textbf{W},{\bf
w}}^{-1}W_{i})<\sum_{i=1}^{m}\dim(V_{i}),$ which is a contradiction,
since $(S_{\textbf{W},{\bf w}}^{-1}W,{\bf v})$ and $({\bf V},{\bf
v})$ are Riesz fusion bases. Thus the conclusion follows.\end{proof}

\begin{rem}
Note that Proposition~\ref{DualRFB}(\ref{CFRB}) implies that if
$(\textbf{W},{\bf w})$ is an orthonormal fusion basis and $({\bf
V},{\bf v})$ is a block-diagonal dual fusion frame which is an
orthonormal fusion basis, then $W_i=V_i$ for $i = 1, \ldots, m.$

Although it could seem surprising that a Riesz fusion basis can have
more than one dual, it is absolutely reasonable taking into account
the generality of the new definition. Moreover observe that
uniqueness of the dual of a Riesz basis in classical frame theory
can be understood in the sense of Proposition~\ref{DualRFB}
(\ref{MC}) considering that each element of the Riesz basis and of
its dual generate subspaces of dimension 1.
\end{rem}

The following result asserts that if $(\textbf{W},{\bf w})$ is an
overcomplete fusion frame (i.e. a fusion frame which is not a Riesz
fusion basis) with non trivial subspaces, there always exist
component preserving dual fusion frames which differ from the
canonical ones.

\begin{prop}
Let $(\textbf{W},{\bf w})$ be an overcomplete fusion frame such that
$W_i\neq \{0\}$ for every $i\in \{1,\dots,m\}.$ Then there exist
component preserving dual fusion frames of $(\textbf{W},{\bf w})$
different from $(S_{\textbf{W},\textbf{w}}^{-1}\textbf{W},{\bf v})$
for any family of weights ${\bf v}.$
\end{prop}

\begin{proof} Since  $(\textbf{W},{\bf w})$ is not a Riesz fusion basis, there exists $i_0\in
\{1,\dots,m\}$ such that $W_{i_0}\cap \text{span}\{W_i:i\neq
i_0\}\neq \{0\}.$ Observe that $(\mathbf{\widetilde{W}},{\bf w})$
given by $\widetilde{W}_i=W_i$ for $i\neq i_0$ and
$\widetilde{W}_{i_0}=W_{i_0}\cap(W_{i_0}\cap\text{span}\{W_i:i\neq
i_0\})^\bot$ is a fusion frame for $\mathcal{H}$. Define
$V_i=S_{\mathbf{\widetilde{W}},{\bf w}}^{-1}\widetilde{W}_{i}$ for
$i\in \{1,\dots,m\}.$ Consider the component preserving
$\widetilde{Q}\in L(\mathcal{W}, \mathcal{V}),$ given by
$\widetilde{Q}(f_i)_{i=1}^m=(S_{\mathbf{\widetilde{W}},{\bf
w}}^{-1}\pi_{\widetilde{W}_i}f_i)_{i=1}^m.$

Let $f\in\mathcal{H}.$ Since
$\pi_{\widetilde{W}_{i_0}}\pi_{W_{i_0}}f=\pi_{\widetilde{W}_{i_0}}f,$
we obtain $T_{\mathbf{V},{\bf w}}\widetilde{Q}T_{\textbf{W},{\bf
w}}^*f=\sum_{i=1}^{m}w_i^2 S_{\mathcal{\widetilde{W}},{\bf
w}}^{-1}\pi_{\widetilde{W}_i}(f)=f.$

Now, $\dim(S_{\mathcal{\widetilde{W}},{\bf
w}}^{-1}\widetilde{W}_{i_0})=\dim(\widetilde{W}_{i_0})$ and
$\widetilde{W}_{i_0}\subseteq W_{i_0}.$ Assume that
$\widetilde{W}_{i_0}=W_{i_0}.$ Then $W_{i_0}\subseteq
(\text{span}\{W_i:i\neq i_0\})^\bot$ which is a contradiction since
$W_{i_0}\cap \text{span}\{W_i:i\neq i_0\}\neq \{0\}.$ Hence
$\dim(\widetilde{W}_{i_0})< \dim(W_{i_0})=\dim( S_{\textbf{W},{\bf
w}}^{-1}W_{i_0})$ and so $V_{i_0}\neq S_{\textbf{W},{\bf
w}}^{-1}W_{i_0}$ obtaining the desired result.\end{proof}

\section{Dual fusion frame systems and their relation with dual frames and dual projective reconstruction systems}

We begin this section by defining dual fusion frame systems. We
first introduce a linear transformation that provides the
fundamental link between the synthesis operator of a fusion frame
system with the synthesis operator of its associated frame.

Let $(\textbf{W}, {\bf w})$ be a Bessel fusion sequence and
$\mathcal{F}_{i}$ be a frame for $W_i$. Let
$$
C_{\mathcal{F}}:\oplus_{i=1}^{m} \mathbb{F}^{|L_{i}|}\rightarrow
\mathcal{W},\,\, C_{\mathcal{F}}((x_i^l)_{l \in
L_i})_{i=1}^{m}=(T_{\mathcal{F}_{i}}(x_i^l)_{l \in
L_i})_{i=1}^{m}.$$

\noindent Then $C_{\mathcal{F}}$ is surjective and
$C_{\mathcal{F}}^{*}: \mathcal{W} \rightarrow \bigoplus_{i=1}^{m}
\mathbb{F}^{|L_{i}|}$,
$C_{\mathcal{F}}^{*}(g_{i})_{i=1}^{m}=(T_{\mathcal{F}_{i}}^{*}g_{i})_{i=1}^{m}$.
The left inverses of $C_{\mathcal{F}}^{*}$ are all
$C_{\widetilde{\mathcal{F}}} \in L(\bigoplus_{i=1}^{m}
\mathbb{F}^{|L_{i}|}, \mathcal{W})$ with
$\widetilde{\mathcal{F}}_{i}$ a dual frame of $\mathcal{F}_{i}$ for
$i=1, \ldots,m$. Note that $$T_{{\bf
w}\mathcal{F}}=T_{\textbf{W},{\bf w}}C_{\mathcal{F}}\,\,\text { and
}\,\,T_{\textbf{W},{\bf w}}=T_{{\bf
w}\mathcal{F}}C_{\widetilde{\mathcal{F}}}^{*}.$$

\begin{defn}\label{D fusion frame system dual} Let $(\textbf{W},{\bf w},
\mathcal{F})$ and $({\bf V},{\bf v}, \mathcal{G})$ be two fusion
frame systems for $\mathcal{H}$ with
$|\mathcal{F}_{i}|=|\mathcal{G}_{i}|$ for $i=1, \ldots,m$. Then
$({\bf V},{\bf v}, \mathcal{G})$ is a dual fusion frame system of
$(\textbf{W},{\bf w}, \mathcal{F})$ if $({\bf V},{\bf v})$ is a
$C_{\mathcal{G}}C_{\mathcal{F}}^{*}$-dual fusion frame of
$(\textbf{W},{\bf w})$.
\end{defn}

\begin{rem}\label{R Q Mi sum Wj subset Mi sum Vj}
Observe that $C_{\mathcal{G}}C_{\mathcal{F}}^{*}: \mathcal{W}
\rightarrow \mathcal{V},\,
C_{\mathcal{G}}C_{\mathcal{F}}^{*}(f_{i})_{i=1}^{m}=(T_{\mathcal{G}_{i}}T_{\mathcal{F}_{i}}^{*}f_{i})_{i=1}^{m}$
is block-diagonal. Whereas, $M_{i}\mathcal{V}\subseteq
C_{\mathcal{G}}C_{\mathcal{F}}^{*}M_{i}\mathcal{W}$ if and only if
$T_{\mathcal{G}_{i}}(R(T_{\mathcal{G}_{i}}^{*})) \subseteq
T_{\mathcal{G}_{i}}(R(T_{\mathcal{F}_{i}}^{*}))$, or equivalently,
$R(T_{\mathcal{G}_{i}}^{*}) =
\pi_{R(T_{\mathcal{G}_{i}}^{*})}R(T_{\mathcal{F}_{i}}^{*})$.
\end{rem}
If in Definition~\ref{D fusion frame system dual}
$C_{\mathcal{G}}C_{\mathcal{F}}^{*}$ is component preserving we say
that $({\bf V},{\bf v}, \mathcal{G})$ is a \emph{component
preserving dual fusion frame system} of $(\textbf{W},{\bf w},
\mathcal{F}).$ Moreover, if in Definition~\ref{D fusion frame system
dual} $({\bf V},{\bf v})$ is a canonical dual fusion frame of
$(\textbf{W},{\bf w})$  we say that $({\bf V},{\bf v}, \mathcal{G})$
is a \emph{canonical dual fusion frame system} of $(\textbf{W},{\bf
w}, \mathcal{F})$.

In the following subsections we show that there is a close relation
of dual fusion frame systems with dual frames and dual projective
reconstruction systems, a fact that supports the idea that
Definition~\ref{D fusion frame system dual} is the proper definition
of dual fusion frame systems. Consequently, this close relation also
reveals that the inclusion of a $``Q"$ in a definition of duality
for fusion frames as in Definition~\ref{D fusion frame dual} is
natural.

\subsection{Dual fusion frame systems and dual frames}

The following theorem gives the link between the concepts of dual
fusion frame system and dual frame.

\begin{thm}\label{T dual fusion frame systems}
Let $(\textbf{W}, {\bf w}), \,(\mathbf{V}, {\bf v})$ be Bessel
fusion sequences, $\mathcal{F}_{i}$ be a frame for $W_i$ and
$\mathcal{G}_{i}$ be a frame for $V_{i}$ with
$|\mathcal{F}_{i}|=|\mathcal{G}_{i}|$ for $i=1, \ldots,m$. The
following conditions are equivalent:
  \begin{enumerate}
    \item ${\bf w}\mathcal{F}$ and ${\bf v}\mathcal{G}$ are dual frames for $\mathcal{H}.$
    \item  $(\mathbf{V}, {\bf v},\mathcal{G})$ is a dual fusion frame system of $(\textbf{W}, {\bf w},\mathcal{F}).$
  \end{enumerate}
\end{thm}

\begin{proof} By Remark~\ref{R wF marco sii WwF fusion frame system} it only
remains to see the duality condition, and this follows from $T_{{\bf
v}\mathcal{G}}T_{{\bf w}\mathcal{F}}^{*}=T_{{\bf V},{\bf
v}}C_{\mathcal{G}}C_{\mathcal{F}}^{*}T_{\textbf{W},{\bf
w}}^{*}$.\end{proof}

The next corollary shows how to construct component preserving dual
fusion frame systems from a given fusion frame using local dual
frames for each subspace and a left inverse of its analysis
operator.

\begin{cor}\label{C A TW T fusion frame system dual}
Let $(\textbf{W}, {\bf w})$ be a fusion frame for $\mathcal{H}$, $A
\in \mathfrak{L}_{T_{\textbf{W},{\bf w}}^{*}}$ and ${\bf v}$ be a
collection of weights. For each $i=1,\ldots,m,$ let
$V_{i}=AM_{i}\mathcal{W}$, $\{f_{i}^{l}\}_{l\in L_i}$ and
$\{\tilde{f}^l_i\}_{l\in L_i}$ be dual frames for $W_i,$ and
$\tilde{\alpha}_i$ and $\tilde{\beta}_i$ frame bounds of
$\{\tilde{f}^l_i\}_{l\in L_i},$ and $\mathcal{G}_i=\{
\frac{1}{v_{i}}A(\chi_{i}(j)\tilde{f}^l_i)_{j=1}^{m}\}_{l\in L_i}$.
Then
\begin{enumerate}
 \item  $\mathcal{G}_i$ is a frame for $V_{i}$ with frame
bounds $\|T_{\textbf{W},{\bf
w}}^{*}\|^{-2}\frac{\widetilde{\alpha}_i}{v_{i}^{2}}$ and
$\|A\|^{2}\frac{\widetilde{\beta}_i}{v_{i}^{2}}.$
 \item $(\mathbf{V},\mathbf{v},\mathcal{G})$ is a component preserving dual fusion frame system of $(\textbf{W}, {\bf w},
 \mathcal{F}).$  In particular, if
$A=S^{-1}_{\mathbf{W},\mathbf{w}}T_{\mathbf{W},\mathbf{w}}$ then
$(\mathbf{V},\mathbf{v},\mathcal{G})$ is a canonical dual fusion
frame system of $(\textbf{W}, {\bf w}, \mathcal{F}).$
\end{enumerate}
\end{cor}

\begin{proof}The set $\{(\chi_{i}(j)\tilde{f}^l_i)_{j=1}^{m}\}_{l\in
L_i}$ is a frame for $M_{i}\mathcal{W}$. Thus, by \cite[Corollary
5.3.2]{Christensen (2003)}, (1) holds.

\noindent If $(h_i)_{i=1}^m \in \mathcal{W}$, then
$Q_{A,\mathbf{v}}(h_i)_{i=1}^m=(\frac{1}{v_{i}}AM_{i}(h_{j})_{j=1}^{m})_{i=1}^m=(
\frac{1}{v_{i}}AM_{i}(\sum_{l\in
L_j}<h_j,f_j^l>\tilde{f}^l_j)_{j=1}^{m})_{i=1}^m=(\sum_{l\in
L_i}<h_i,f_i^l>
\frac{1}{v_{i}}A(\chi_{i}(j)\tilde{f}^l_i)_{j=1}^{m})_{i=1}^m=C_{\mathcal{G}}C_{\mathcal{F}}^{*}(h_i)_{i=1}^m$.
So (2) follows from (1) and Theorem~\ref{T V,v dual fusion frame sii
Vi=ApiWj}.\end{proof}

The next result exhibits a way to construct component preserving
dual fusion frame systems from a given frame using a left inverse of
its analysis operator.

\begin{cor}\label{C A TF T fusion frame system dual}
Let ${\bf w}$ and ${\bf v}$ be two collections of weights. Let ${\bf
w}\mathcal{F}$ be a frame for $\mathcal{H}$ with local frame bounds
$\alpha_i$, $\beta_i$, $A \in \mathfrak{L}_{T_{{\bf
w}\mathcal{F}}^{*}}$ and $\{\{e_{i}^{l}\}_{l\in L_i}\}_{i=1}^m$ be
the standard basis for $\oplus_{i=1}^{m}\mathbb{F}^{|L_{i}|}$. For
each $i \in \{1,\ldots,m\}$, let
$W_{i}=\text{span}\{f_{i}^{l}\}_{l\in L_i}$ and
$V_{i}=\text{span}\{\frac{1}{v_{i}}Ae_{i}^{l}\}_{l\in L_i}$. Set
$\mathcal{G}=\{\{\frac{1}{v_{i}}Ae^l_i\}_{l\in L_i}\}_{i=1}^{m}.$
Then
\begin{enumerate}
 \item $\{\frac{1}{v_{i}}Ae_{i}^{l}\}_{l\in L_i}$ is a frame for $V_{i}$ with frame
bounds $\|T_{{\bf
w}\mathcal{F}}^{*}\|^{-2}\frac{\alpha_i}{v_{i}^{2}}$ and
$\|A\|^{2}\frac{\beta_i}{v_{i}^{2}}$.
 \item $(\mathbf{V}, {\bf v}, \mathcal{G})$ is a dual fusion frame system of $(\textbf{W}, {\bf w}, \mathcal{F}).$
\end{enumerate}
\end{cor}

\begin{proof}Part (1) follows from \cite[Corollary 5.3.2]{Christensen
(2003)}. By \cite[Lemma 5.6.3]{Christensen (2003)}, ${\bf
w}\mathcal{F}$ and ${\bf v}\mathcal{G}$ are dual frames for
$\mathcal{H}$. So, part (2) follows from Theorem~\ref{T dual fusion
frame systems}.\end{proof}

Let $f \in \mathcal{H}$. For a fusion frame system $(\textbf{W},
{\bf w},\mathcal{F})$ for $\mathcal{H}$ with local frame bounds
$\alpha,\,\beta$ and associated local dual frames
$\{\tilde{f}_i^l\}_{l\in L_i},\, i=1,\ldots,m,$ in
\cite{Casazza-Kutyniok-Li (2008)}  it is considered the
\textit{centralized reconstruction}
\begin{equation}\label{cent}
f=\sum_{i=1}^m \sum_{l\in L_i}\langle f,w_if_i^l\rangle (S_{{\bf w}
\mathcal{F}}^{-1}w_if_i^l)
\end{equation} and the \textit{distributed reconstruction}
\begin{equation}\label{distr}
f=\sum_{i=1}^m \sum_{l\in L_i}\langle f,w_if_i^l\rangle
(S_{\textbf{W},\bf w}^{-1}w_i \tilde{f}_i^l).
\end{equation}

Let now $({\bf V},{\bf v}, \mathcal{G})$ be any dual fusion frame
system of $(\textbf{W}, {\bf w},\mathcal{F}).$ We have the
reconstruction formula
\begin{equation}\label{recgral}
f=\sum_{i=1}^m \sum_{l\in L_i}\langle f,w_if_i^l\rangle v_ig_i^l.
\end{equation}
By Corollary~\ref{C A TF T fusion frame system dual}  with
$A=S_{{\bf w}\mathcal{F}}^{-1}T_{{\bf w}\mathcal{F}}$, $(\textbf{W},
{\bf w},\mathcal{F})$ and $(S_{{\bf w}\mathcal{F}}^{-1}\textbf{W},
1,S_{\bf w\mathcal{F}}^{-1}{\bf w}\mathcal{F})$ are dual fusion
frame systems for $\mathcal{H},$ and hence (\ref{cent}) turns out to
be a particular case of (\ref{recgral}). On the other hand, by
Corollary~\ref{C A TW T fusion frame system dual}  with
$A=S_{\textbf{W},{\bf w}}^{-1}T_{\textbf{W},{\bf w}}$, $(\textbf{W},
{\bf w},\mathcal{F})$ and $(S_{\textbf{W},{\bf
w}}^{-1}\textbf{W},1,S_{\textbf{W},{\bf w}}^{-1} {\bf
w}\mathcal{\widetilde{F}})$ are dual fusion frame systems for
$\mathcal{H},$ and so (\ref{distr}) can also be seen as a particular
case of (\ref{recgral}). The advantage of reconstruction
(\ref{recgral}) is that now we can give many different
representations of $f$ according to our needs, since we have more
freedom for the choice of $\{g_i^l\}_{l\in L_i},\,i=1,\ldots, m.$

\subsection{Dual fusion frame systems and dual projective reconstruction systems}

The concept of reconstruction systems for finite-dimensional Hilbert
spaces was introduced in \cite{Massey (2009)}. Previously in
\cite{Sun (2006)} reconstruction systems for any separable Hilbert
spaces were called g-frames.

\begin{defn}\label{D reconstruction system}
Let $T_{i} \in L(\mathbb{F}^{n_{i}},\mathcal{H})$ for $i=1,\ldots
m$.
\begin{enumerate}
  \item The \textit{synthesis operator} of $(T_{i})_{i=1}^{m}$ is $T : \bigoplus_{i=1}^{m} \mathbb{F}^{n_{i}}\rightarrow \mathcal{H}$,
$T(x_{i})_{i=1}^{m}=\sum_{i=1}^{m}T_{i}x_{i}$ and the
\textit{analysis operator} is $T^{*} : \mathcal{H} \rightarrow
\bigoplus_{i=1}^{m} \mathbb{F}^{n_{i}}$,
$T^{*}f=(T_{i}^{*}f)_{i=1}^{m}.$
  \item The sequence
$(T_{i})_{i=1}^{m}$ is an
$(m,\mathbf{n},\mathcal{H})$-\textit{reconstruction system} if
$\text{span}\cup_{i=1}^{m}R(T_{i})=\mathcal{H}.$
  \item In case $(T_{i})_{i=1}^{m}$ is an
$(m,\mathbf{n},\mathcal{H})$-reconstruction system,
$S=\sum_{i=1}^{m}T_{i}T_{i}^{*}$ is called the
\textit{reconstruction system operator} of $(T_{i})_{i=1}^{m}$.
\end{enumerate}
\end{defn}
An $(m,1,\mathcal{H})$-reconstruction system is a frame. The set of
$(m,\mathbf{n},\mathcal{H})$-reconstruction systems is denoted with
$\mathcal{RS}(m,\mathbf{n},\mathcal{H})$. If $n_{i}=n$ for
$i=1,\ldots m,$ we write $(m,n,\mathcal{H})$-reconstruction system.

\begin{defn}(cf. \cite[Definition 2.5]{ Massey-Ruiz-Stojanoff (2012a)})\label{D reconstruction system dual} Let
$(T_{i})_{i=1}^{m}, (\widetilde{T}_{i})_{i=1}^{m} \in
\mathcal{RS}(m,\mathbf{n},\mathcal{H})$. Then $(T_{i})_{i=1}^{m}$
and $(\widetilde{T}_{i})_{i=1}^{m}$ are dual reconstruction systems
if $\widetilde{T}T^{*}=I_{\mathcal{H}}$.
\end{defn}

In \cite{Massey-Ruiz-Stojanoff (2012a)} the relation between
reconstruction systems and fusion frames is established via
\textit{projective} reconstruction systems.

\begin{defn} $(T_{i})_{i=1}^{m} \in
\mathcal{RS}(m,\mathbf{n},\mathcal{H})$, is said to be
\textit{projective} if there exists a sequence of weights
$\mathbf{w}=(w_{i})_{i=1}^{m} \in \mathbb{R}_{+}^{m}$ such that
$T_{i}^{*}T_{i}=w_{i}^{2}I_{\mathbb{F}^{n_{i}}} \text{, } i = 1,
\ldots, m.$
\end{defn}

If $(T_{i})_{i=1}^{m} \in \mathcal{RS}(m,\mathbf{n},\mathcal{H})$ is
projective, then $w_{i}=\|T_{i}\|_{sp}$,
$T_{i}T_{i}^{*}=w_{i}^{2}P_{R(T_{i})}$,
$S=\sum_{i=1}^{m}w_{i}^{2}P_{R(T_{i})}$ and
$\{(R(T_{i}),\|T_{i}\|_{sp})\}_{i=1}^m$ is a fusion frame for
$\mathcal{H}$. Conversely, if $(\textbf{W},\mathbf{w})$ is a fusion
frame for $\mathcal{H}$, then there exists a (non unique) projective
$(T_{i})_{i=1}^{m} \in \mathcal{RS}(m,\mathbf{n},\mathcal{H})$ such
that $R(T_{i})=W_i$ and $\|T_{i}\|_{sp}=w_{i}$.

The next corollary gives the relation between dual fusion frame
systems and dual reconstruction systems in the projective case.

\begin{cor}\label{C T dual fusion frame systems projective reconstruction system}
Let $(\textbf{W},\mathbf{w})$ and $(\mathbf{V},\mathbf{v})$ be
fusion frames for $\mathcal{H}$, and $(T_{i})_{i=1}^{m}$ and
$(\widetilde{T}_{i})_{i=1}^{m}$ be projective
$(m,\mathbf{n},\mathcal{H})$-reconstruction systems for
$\mathcal{H}$ such that $R(T_{i})=W_i$, $\|T_{i}\|_{sp}=w_{i}$, and
$R(\widetilde{T}_{i})=V_{i}$ and $\|\widetilde{T}_{i}\|_{sp}=v_{i}$,
respectively. For $\{e_{n_{i}}^l\}_{l = 1}^{n_{i}}$  an orthonormal
basis for $\mathbb{F}^{n_{i}},$ set
$\mathcal{F}_{i}=\{\frac{1}{w_{i}}T_{i}e_{n_{i}}^l\}_{l= 1}^{n_{i}}$
and
$\mathcal{G}_{i}=\{\frac{1}{v_{i}}\widetilde{T}_{i}e_{n_{i}}^l\}_{l=
1}^{n_{i}}$ for $l=1,\ldots,n_i.$ Then the following conditions are
equivalent.
\begin{enumerate}
  \item $(\mathbf{V},\mathbf{v}, \mathcal{G})$ is a dual fusion frame system of
  $(\textbf{W},\mathbf{w},\mathcal{F}).$
  \item $(\widetilde{T}_{i})_{i=1}^{m}$ is a dual
  $(m,\mathbf{n},\mathcal{H})$-reconstruction system of
  $(T_{i})_{i=1}^{m}$.
\end{enumerate}
\end{cor}

\begin{proof}We have that $\mathcal{F}_{i}$ is a frame for $W_i$,
$\mathcal{G}_{i}$  is a frame for $V_{i}$ and
$\sum_{i=1}^{m}\widetilde{T}_{i}T_{i}^{*}=T_{w\mathcal{G}}T_{w\mathcal{F}}^{*}.$
Then the conclusion follows from Definition~\ref{D reconstruction
system dual}, Definition~\ref{D frame dual} and Theorem~\ref{T dual
fusion frame systems}.\end{proof}

In view of the relation between fusion frames and projective
reconstruction systems, the study of duality of fusion frames can be
done in the context of projective reconstruction systems using
Definition~\ref{D reconstruction system dual}. This approach is
considered in \cite{Massey-Ruiz-Stojanoff (2012a)} and
\cite{Massey-Ruiz-Stojanoff (2012b)}. But a dual reconstruction
system of a projective reconstruction system is not always
projective. In \cite{Massey-Ruiz-Stojanoff (2012b)} the authors
provide examples of projective reconstruction systems with non
projective canonical dual or without projective duals at all. These
projective reconstruction systems, along with their associated
fusion frames, are considered in examples \ref{Ej
MasseyRuizStojanoff 2} and \ref{Ej MasseyRuizStojanoff 1} below. It
is worth to note that with Definition~\ref{D fusion frame dual} the
dual of a fusion frame is always a fusion frame. Moreover, as it was
shown in subsection 3.1, a fusion frame has always a canonical dual
fusion frame. Similar considerations for fusion frame systems follow
from Definition~\ref{D fusion frame system dual} and
Corollary~\ref{C A TW T fusion frame system dual}.

\section{Optimal dual fusion frames for erasures}\label{S ODFFE}

Having different duals is convenient in many applications e.g. in
the theory of \textit{optimal dual fusion frames for erasures} that
will be discussed in this section. In this case, (\ref{E TvQTw*=I})
(or (\ref{recgral}) ) can give a reconstruction that behaves better
than the ones given in (\ref{E TvQTw*=I canonico}) ((\ref{cent})  or
(\ref{distr})).

Let $(\textbf{W},\mathbf{w})$ be a fusion frame for $\mathcal{H}$.
In applications an element $f \in \mathcal{H}$ (e. g. a signal) is
converted into the data vectors $T_{\textbf{W},\mathbf{w}}^{*}f$. In
an ideal setting these vectors are transmitted and $f$ can be
reconstructed by the receiver using
$f=T_{\mathbf{V},\mathbf{v}}QT_{\textbf{W},\mathbf{w}}^{*}f,$ where
$(\mathbf{V},\mathbf{v})$ is some $Q$-dual of
$(\textbf{W},\mathbf{w})$. But in real implementations, sometimes
some of the data vectors, or part of them, are lost or erased, and
it is necessary to reconstruct $f$ with the partial information at
hand. There are several approaches to study this problem, here we
consider optimal dual fusion frames for a fixed fusion frame when a
blind reconstruction process is used, in a similar way as in
\cite{Lopez-Han (2010), Leng-Han (2011)} for frames and in
\cite{Massey-Ruiz-Stojanoff (2012b)} for projective reconstruction
systems.

\subsection{Optimal dual fusion frames for erasures of subspaces}\label{S ODFFE subspaces}

Let $J \subseteq \{1, \ldots, m\}$ and suppose that the data vectors
corresponding to the subspaces $\{W_{j}\}_{j \in J}$ are lost. The
reconstruction then gives $T_{\mathbf{V},\mathbf{v}}QM_{\{1, \ldots,
m\} \setminus J}T_{\textbf{W},\mathbf{w}}^{*}f$. So we need to find
those dual fusion frames of $(\textbf{W},\mathbf{w})$ that are in
some sense optimal for this situation.

Fix $r \in \{1, \ldots, m\}$. Let $\mathcal{P}_{r}^{m}:=\{J
\subseteq \{1, \ldots, m\} : |J|=r\}$. Noting that
$M_{J}=I_{\mathcal{W}}-M_{\{1, \ldots, m\} \setminus J}$, given a
$Q$-dual fusion frame $(\mathbf{V},\mathbf{v})$ of
$(\textbf{W},\mathbf{w})$ we consider the vector error

\centerline{$e(r,(\textbf{W},\mathbf{w}),(\mathbf{V},\mathbf{v}),Q)=(\|T_{\mathbf{V},\mathbf{v}}QM_{J}T_{\textbf{W},\mathbf{w}}^{*}\|_{F})_{J
\in \mathcal{P}_{r}^{m}}.$}

For $p\in\mathbb{N}\cup\{\infty\}$ we define inductively:

\centerline{$e_{1}^{(p)}(\textbf{W},\mathbf{w})=\text{inf}\{\|e(1,(\textbf{W},\mathbf{w}),(\mathbf{V},\mathbf{v}),Q)\|_{p}:
(\mathbf{V},\mathbf{v}) ~\text{is a $Q$-dual fusion frame of}~
(\textbf{W},\mathbf{w})\},$}

$\mathcal{D}_{1}^{(p)}(\textbf{W},\mathbf{w})$ as the set of
$((\mathbf{V},\mathbf{v}),Q)$ where $(\mathbf{V},\mathbf{v})$ is a
$Q$-dual fusion frame of $(\textbf{W},\mathbf{w})$ and

\centerline{$\|e(1,(\textbf{W},\mathbf{w}),(\mathbf{V},\mathbf{v}),Q)\|_{p}=e_{1}^{(p)}(\textbf{W},\mathbf{w}),$}

\centerline{$e_{r}^{(p)}(\textbf{W},\mathbf{w})=\text{inf}\{
\|e(r,(\textbf{W},\mathbf{w}),(\mathbf{V},\mathbf{v}),Q)\|_{p}:((\mathbf{V},\mathbf{v}),Q)
\in \mathcal{D}_{r-1}^{(p)}(\textbf{W},\mathbf{w})\},$}

\centerline{$\mathcal{D}_{r}^{(p)}(\textbf{W},\mathbf{w})=\{((\mathbf{V},\mathbf{v}),Q)
\in
\mathcal{D}_{r-1}^{(p)}(\textbf{W},\mathbf{w}):\|e(r,(\textbf{W},\mathbf{w}),(\mathbf{V},\mathbf{v}),Q)\|_{p}
=e_{r}^{(p)}(\textbf{W},\mathbf{w})\},$}

\noindent in case each
$\mathcal{D}_{r}^{(p)}(\textbf{W},\mathbf{w})$, called the set of
$(r,p)$-loss optimal dual fusion frames for
$(\textbf{W},\mathbf{w})$, is non empty.

Note that
$M_{i}T_{\textbf{W},\mathbf{w}}^{*}T_{\textbf{W},\mathbf{w}}M_{i}^{*}=w_{i}^{2}M_{i}$.
Let $A \in L(\mathcal{W},\mathcal{H})$, then
\begin{equation}\label{E T mse 1a}
\|AM_{i}T_{\textbf{W},\mathbf{w}}^{*}\|_{F}^{2}=w_{i}^{2}\|AM_{i}\|_{F}^{2}.
\end{equation}

\subsubsection{The mean square error}\label{S ODFFE subspaces p=2}

Consider the mean square error,

\centerline{$\|e(r,(\textbf{W},\mathbf{w}),(\mathbf{V},\mathbf{v}),Q)\|_{2}=(\sum_{J
\in
\mathcal{P}_{r}^{m}}\|T_{\mathbf{V},\mathbf{v}}QM_{J}T_{\textbf{W},\mathbf{w}}^{*}\|_{F}^{2})^{1/2}.$}

The next theorem describes a $(r,2)$-loss optimal component
preserving dual fusion frame of a given fusion frame. It also
asserts that the reconstruction formula provided by this dual
coincides with the reconstruction formula provided by any other
$(r,2)$-loss optimal dual fusion frame. Furthermore, it shows that
it is the only $(r,2)$-loss optimal component preserving dual fusion
frame and when it coincides with a canonical dual.

\begin{thm}\label{T ODFF p=2}
Let $(\textbf{W},\mathbf{w})$ be a fusion frame for $\mathcal{H}$.
Let $D: \mathcal{W} \rightarrow \mathcal{W}$,
$D(f_i)_{i=1}^{m}=(\frac{1}{w_{i}v_{i}}f_{i})_{i=1}^{m}$ and
$S_{D}=T_{\textbf{W},\mathbf{v}}DT_{\textbf{W},\mathbf{w}}^{*}$.
Then $S_{D}$ is a selfadjoint positive invertible operator and if
$Q_{D}: \mathcal{W} \rightarrow \bigoplus_{j=1}^{m}
S_{D}^{-1}W_{j}$,
$Q_{D}(f_j)_{j=1}^{m}=(\frac{1}{w_{j}v_{j}}S_{D}^{-1}f_{j})_{j=1}^{m}$
then
\begin{enumerate}
  \item $(S_{D}^{-1}\textbf{W},\mathbf{v})$ is a $(r,2)$-loss optimal
component preserving $Q_{D}$-dual for $(\textbf{W},\mathbf{w})$.
  \item If $(\mathbf{V},\mathbf{v})$ is
a $(r,2)$-loss optimal $Q$-dual fusion frame of
$(\textbf{W},\mathbf{w})$, then
$T_{\mathbf{V},\mathbf{v}}Q=T_{S_{D}^{-1}\textbf{W},\mathbf{v}}Q_{D}$.
  \item If
$(\mathbf{V},\mathbf{v})$ is a $(r,2)$-loss optimal $Q$-component
preserving dual fusion frame of $(\textbf{W},\mathbf{w})$ then
$\mathbf{V}=S_{D}^{-1}\textbf{W}$ and $Q=Q_{D}$.
  \item If $w_{i}=w$ for $i = 1, \ldots, m$, then   $S_{D}^{-1}\textbf{W}=S_{\textbf{W},w}^{-1}\textbf{W}$ and $Q_{D}=Q_{S_{\textbf{W},w}^{-1}T_{\textbf{W},w},\mathbf{v}}$.
\end{enumerate}
\end{thm}

\begin{proof}Clearly, $S_{D}=T_{\textbf{W},\mathbf{v}}DT_{\textbf{W},\mathbf{w}}^{*}=\sum_{i=1}^{m}\pi_{W_{i}}$
is selfadjoint.

Let $f \in \mathcal{H}$. If $\alpha > 0$ is the lower fusion frame
bound of $(\textbf{W},\mathbf{w})$, then $$\langle
S_{D}f,f\rangle=\sum_{i=1}^{m}w_{i}^{-2}w_{i}^{2}\langle\pi_{W_{i}}f,f\rangle
\geq(\min_{i \in \{1, \ldots, m\}}w_{i}^{-2})\langle
S_{\textbf{W},\mathbf{w}}f,f\rangle \geq (\min_{i \in \{1, \ldots,
m\}}w_{i}^{-2}) \alpha \|f\|^{2}.$$ So $S_{D}$ is positive and
invertible. Since $S_{D}^{-1}$ is linear, it is easy to see that
$Q_{D}$ is component preserving.

We have

\centerline{$T_{S_{D}^{-1}\textbf{W},\mathbf{v}}Q_{D}T_{\textbf{W},\mathbf{w}}^{*}=\sum_{i=1}^{m}S_{D}^{-1}\pi_{W_{i}}=S_{D}^{-1}S_{D}=I_{\mathcal{H}},$}

\noindent thus $(S_{D}^{-1}\textbf{W},\mathbf{v})$ is a
$Q_{D}$-component preserving dual fusion frame of
$(\textbf{W},\mathbf{w})$.

Let $(\mathbf{V},\mathbf{v})$ be a fusion frame for $\mathcal{H}$.
Using (\ref{E T mse 1a}),
\begin{align}
\|T_{\mathbf{V},\mathbf{v}}QM_{i}T_{\textbf{W},\mathbf{w}}^{*}\|_{F}^{2}=&w_{i}^{2}\|T_{\mathbf{V},\mathbf{v}}QM_{i}\|_{F}^{2}\notag\\
=&w_{i}^{2}\|T_{S_{D}^{-1}\textbf{W},\mathbf{v}}Q_{D}M_{i}\|_{F}^{2}+w_{i}^{2}\|T_{\mathbf{V},\mathbf{v}}QM_{i}-T_{S_{D}^{-1}\textbf{W},
\mathbf{v}}Q_{D}M_{i}\|_{F}^{2}\notag\\
&+2\text{Re}(w_{i}^{2}\text{tr}[(T_{\mathbf{V},\mathbf{v}}QM_{i}-T_{S_{D}^{-1}\textbf{W},\mathbf{v}}Q_{D}M_{i})M_{i}^{*}Q_{D}^{*}T_{S_{D}^{-1}\textbf{W},\mathbf{v}}^{*}])\label{E
T mse 1}.
\end{align}
Suppose that $(\mathbf{V},\mathbf{v})$ is a $Q$-dual fusion frame of
$(\textbf{W},\mathbf{w})$. Using
$T_{S_{D}^{-1}\textbf{W},\mathbf{v}}Q_{D}M_{i}=w_{i}^{-2}S_{D}^{-1}T_{\textbf{W},\mathbf{w}}M_{i}$
and
$T_{\mathbf{V},\mathbf{v}}QT_{\textbf{W},\mathbf{w}}^{*}=T_{S_{D}^{-1}\textbf{W},\mathbf{v}}Q_{D}T_{\textbf{W},\mathbf{w}}^{*}=I_{\mathcal{H}}$
we have
\begin{align}\label{E T mse 2}
\sum_{i=1}^{m}w_{i}^{2}\text{tr}[(T_{\mathbf{V},\mathbf{v}}QM_{i}-&T_{S_{D}^{-1}\textbf{W},\mathbf{v}}Q_{D}M_{i})M_{i}^{*}Q_{D}^{*}T_{S_{D}^{-1}\textbf{W},\mathbf{v}}^{*}]=\notag\\
&=\sum_{i=1}^{m}w_{i}^{2}w_{i}^{-2}\text{tr}[(T_{\mathbf{V},\mathbf{v}}QM_{i}-T_{S_{D}^{-1}\textbf{W},\mathbf{v}}Q_{D}M_{i})M_{i}^{*}T_{\textbf{W},\mathbf{w}}^{*}S_{D}^{-1}]\notag\\
&=\text{tr}[(T_{\mathbf{V},\mathbf{v}}Q-T_{S_{D}^{-1}\textbf{W},\mathbf{v}}Q_{D})T_{\textbf{W},\mathbf{w}}^{*}S_{D}^{-1}]=0.
\end{align}
By (\ref{E T mse 1}), (\ref{E T mse 1a}) and (\ref{E T mse 2}) we
obtain
\begin{align}\label{E T mse 3}
\sum_{i=1}^{m}\|T_{\mathbf{V},\mathbf{v}}QM_{i}T_{\textbf{W},\mathbf{w}}^{*}\|_{F}^{2}
=&\sum_{i=1}^{m}\|T_{S_{D}^{-1}\textbf{W},\mathbf{v}}Q_{D}M_{i}T_{\textbf{W},\mathbf{w}}^{*}\|_{F}^{2}\notag\\
&+\sum_{i=1}^{m}w_{i}^{2}\|T_{\mathbf{V},\mathbf{v}}QM_{i}-T_{S_{D}^{-1}\textbf{W},\mathbf{v}}Q_{D}M_{i}\|_{F}^{2}.
\end{align}Thus,
$\|e(1,(\textbf{W},\mathbf{w}),(\mathbf{V},\mathbf{v}),Q)\|_{2} \geq
\|e(1,(\textbf{W},\mathbf{w}),(S_{D}^{-1}\textbf{W},\mathbf{v}),Q_{D})\|_{2}$.
So $(S_{D}^{-1}\textbf{W},\mathbf{v}) \in
\mathcal{D}_{1}^{(2)}(\textbf{W},\mathbf{w})$.

If $(\mathbf{V},\mathbf{v}) \in
\mathcal{D}_{1}^{(2)}(\textbf{W},\mathbf{w}),$ then
$\|e(1,(\textbf{W},\mathbf{w}),(\mathbf{V},\mathbf{v}),Q)\|_{2} =
\|e(1,(\textbf{W},\mathbf{w}),(S_{D}^{-1}\textbf{W},\mathbf{v}),Q_{D})\|_{2},$
and by (\ref{E T mse 3})
\begin{equation}\label{E T mse 4}
T_{\mathbf{V},\mathbf{v}}Q=T_{S_{D}^{-1}\textbf{W},\mathbf{v}}Q_{D}.
\end{equation}
Suppose now that $Q$ is component preserving. By (\ref{E T mse 4})
and Lemma~\ref{L W V Vtild Q Qtild comp preserv},
$\mathbf{V}=S_{D}^{-1}\textbf{W}$ and $Q=Q_{D}$.

By the hierarchical definitions of
$\mathcal{D}_{r}^{(p)}(\textbf{W},\mathbf{w})$ for $r \geq 1$, the
conclusions (1)-(3) follow.

If $w_{i}=w$ for $i = 1, \ldots, m$, then
$S_{D}=w^{-2}S_{\textbf{W},w}$. Thus, (4) follows.\end{proof}

\subsubsection{The worst-case error}\label{S ODFFE subspaces p = infty}

Consider $p=\infty$. In this case, we obtain the worst-case error

\centerline{$\|e(r,(\textbf{W},\mathbf{w}),(\mathbf{V},\mathbf{v}),Q)\|_{\infty}=\max_{J
\in
\mathcal{P}_{r}^{m}}\|T_{\mathbf{V},\mathbf{v}}QM_{J,\textbf{W}}T_{\textbf{W},\mathbf{w}}^{*}\|_{F}.$}

To prove Theorem~\ref{T ODFFE p infty canonico} we need the
following proposition that gives some properties of the set of
elements in $\mathfrak{L}_{T_{\textbf{W},\mathbf{w}}^{*}}$
satisfying certain optimality condition.

\begin{prop}\label{P W w set duals}
Let $(\textbf{W},\mathbf{w})$ be a fusion frame for $\mathcal{H}$.
Then
\begin{equation}\label{E P W w set duals}\{A \in
\mathfrak{L}_{T_{\textbf{W},\mathbf{w}}^{*}}: \max_{1 \leq i \leq
m}\|AM_{i}T_{\textbf{W},\mathbf{w}}^{*}\|_{F} = \min_{B \in
\mathfrak{L}_{T_{\textbf{W},\mathbf{w}}^{*}}}\max_{1 \leq i \leq
m}\|BM_{i}T_{\textbf{W},\mathbf{w}}^{*}\|_{F}\}\end{equation} is non
empty, compact and convex.
\end{prop}

\begin{proof}The map $\|.\|_{\textbf{W},\mathbf{w}}: L(\mathcal{W},\mathcal{H})  \rightarrow \mathbb{R}^{+}$,
$\|A\|_{\textbf{W},\mathbf{w}}=\max_{1 \leq i \leq
m}\|AM_{i}T_{\textbf{W},\mathbf{w}}^{*}\|_{F}$ \noindent is a norm
in $L(\mathcal{W},\mathcal{H})$. To see this, let $A \in
L(\mathcal{W},\mathcal{H})$ such that $\|A\|_{\textbf{W},\mathbf{w}}
= 0$. Then $AM_{i}T_{\textbf{W},\mathbf{w}}^{*}=0$ for $i = 1,
\ldots, m$, and since
$\text{R}(M_{i}T_{\textbf{W},\mathbf{w}}^{*})=M_i\mathcal{W}$, it
follows that $AM_i\mathcal{W}=\{0\}$ for $i = 1, \ldots, m$. Thus
$A=\sum_{i=1}^{m}AM_i=0$. The other norm properties are immediate.

Since the set $\mathfrak{L}_{T_{\textbf{W},\mathbf{w}}^{*}}$ is
closed in $L(\mathcal{W},\mathcal{H})$ under the usual norm and all
norms in a finite-dimensional Hilbert space are equivalent,
$\mathfrak{L}_{T_{\textbf{W},\mathbf{w}}^{*}}$ is a closed subset of
$L(\mathcal{W},\mathcal{H})$ under the norm
$\|.\|_{\textbf{W},\mathbf{w}}$. Given $B_{0} \in
\mathfrak{L}_{T_{\textbf{W},\mathbf{w}}^{*}}$, $B_{0} \neq 0$, there
exists an $A_{0}$ in the non empty compact set $\{A \in
\mathfrak{L}_{T_{\textbf{W},\mathbf{w}}^{*}}:
\|A\|_{\textbf{W},\mathbf{w}} \leq
\|B_{0}\|_{\textbf{W},\mathbf{w}}\}$ where the continuous map
$\|.\|_{\textbf{W},\mathbf{w}}$ attains its minimum. So,
$\|A_{0}\|_{\textbf{W},\mathbf{w}}=\min_{A \in
\mathfrak{L}_{T_{\textbf{W},\mathbf{w}}^{*}}}\|A\|_{\textbf{W},\mathbf{w}}$,
and the set (\ref{E P W w set duals}) is non empty and compact.

Since $\mathfrak{L}_{T_{\textbf{W},\mathbf{w}}^{*}}$ is convex and
$\|.\|_{\textbf{W},\mathbf{w}}$ is a convex map, the set (\ref{E P W
w set duals}) is convex.\end{proof}

Theorem~\ref{T ODFFE p infty canonico} gives sufficient conditions
that assure that the only $(r,\infty)$-loss optimal component
preserving dual fusion frames are the canonical ones.

\begin{thm}\label{T ODFFE p infty canonico}
Let $(\textbf{W},\mathbf{w})$ be a fusion frame for $\mathcal{H}$.
If $w_{i}^{2}\|S_{\textbf{W},\mathbf{w}}^{-1}\pi_{W_{i}}\|_{F}=c$
for each $i = 1, \ldots, m$, then
\begin{enumerate}
  \item $S_{\textbf{W},\mathbf{w}}^{-1}T_{\textbf{W},\mathbf{w}}$ is the unique
  element of the set (\ref{E P W w set
duals}).
  \item The only $(r,\infty)$-loss optimal component preserving dual fusion frames of $(\textbf{W},\mathbf{w})$ are the canonical ones $(S_{\textbf{W},\mathbf{w}}^{-1}\textbf{W},\mathbf{v})$ with arbitrary vector of weights $\mathbf{v}$.
\end{enumerate}
\end{thm}

\begin{proof}
Let $A_{0}=S_{\textbf{W},\mathbf{w}}^{-1}T_{\textbf{W},\mathbf{w}}$.

(1) By Proposition~\ref{P W w set duals}, there exists $A \in
\mathfrak{L}_{T_{\textbf{W},\mathbf{w}}^{*}}$ such that

\centerline{$\max_{1 \leq i \leq
m}\|AM_{i}T_{\textbf{W},\mathbf{w}}^{*}\|_{F} = \min_{B \in
\mathfrak{L}_{T_{\textbf{W},\mathbf{w}}^{*}}}\max_{1 \leq i \leq
m}\|BM_{i}T_{\textbf{W},\mathbf{w}}^{*}\|_{F}.$}

So,

\centerline{$\max_{1 \leq i \leq
m}\|AM_{i}T_{\textbf{W},\mathbf{w}}^{*}\|_{F} \leq \max_{1 \leq i
\leq m}\|A_{0}M_{i}T_{\textbf{W},\mathbf{w}}^{*}\|_{F},$}

\noindent and then, by hypothesis and (\ref{E T mse 1a}),
\begin{equation}\label{E T 1 infty opt dual c 1}
\|AM_{i}\|_{F} \leq \|A_{0}M_{i}\|_{F},~~\text{for each $i \in \{1,
\ldots, m\}$.}
\end{equation}
Using
$\|AM_{i}\|_{F}^{2}=\|A_{0}M_{i}\|_{F}^{2}+\|(A-A_{0})M_{i}\|_{F}^{2}+2\text{Re}(\text{tr}[(A-A_{0})M_{i}
T_{\textbf{W},\mathbf{w}}^{*}S_{\textbf{W},\mathbf{w}}^{-1}]),$ by
(\ref{E T 1 infty opt dual c 1}),
\begin{equation}\label{E T 1 infty opt dual c 2}
\|(A-A_{0})M_{i}\|_{F}^{2} +
2\text{Re}(\text{tr}[(A-A_{0})M_{i}T_{\textbf{W},\mathbf{w}}^{*}S_{\textbf{W},\mathbf{w}}^{-1}])
\leq 0.
\end{equation}
Since
$AT_{\textbf{W},\mathbf{w}}^{*}=A_{0}T_{\textbf{W},\mathbf{w}}^{*}=I_{\mathcal{H}}$,

\centerline{$\sum_{i=1}^{m}\text{tr}[(A-A_{0})M_{i}T_{\textbf{W},\mathbf{w}}^{*}S_{\textbf{W},\mathbf{w}}^{-1}]=\text{tr}[(A-A_{0})T_{\textbf{W},\mathbf{w}}^{*}S_{\textbf{W},\mathbf{w}}^{-1}]=0.$}

\noindent Thus, by (\ref{E T 1 infty opt dual c 2}),
$\sum_{i=1}^{m}\|(A-A_{0})M_{i}\|_{F}^{2} \leq 0$, and consequently,
$AM_{i}=A_{0}M_{i}$ for every $i \in \{1, \ldots, m\}$, or
equivalently, $A=A_{0}$.

(2) It follows from part (1), Theorem~\ref{T V,v dual fusion frame
sii Vi=ApiWj}, Lemma~\ref{L W V Vtild Q Qtild comp preserv} and the
inductive definition of $(r,\infty)$-loss optimal dual fusion
frames.\end{proof}

\begin{cor}\label{C T 1 infty opt dual c 1}
Let $(\textbf{W},\mathbf{w})$ be a Parseval fusion frame for
$\mathcal{H}$. If $w_{i}^{2}\dim(W_{i})^{1/2}=c$ for each $i = 1,
\ldots, m$, then
\begin{enumerate}
  \item $T_{\textbf{W},\mathbf{w}}$ is the unique
  element of the set (\ref{E P W w set
duals}).
  \item The only $(r,\infty)$-loss optimal component preserving dual fusion frames of $(\textbf{W},\mathbf{w})$ are the canonical ones $(\textbf{W},\mathbf{v})$ with arbitrary vector of weights $\textbf{v}$.
\end{enumerate}
\end{cor}

\begin{proof} By hypothesis, $\|T_{\textbf{W},\mathbf{w}}M_{i}T_{\textbf{W},\mathbf{w}}^{*}\|_{F}=w_{i}^{2}\|\pi_{W_{i}}\|_{F}=w_{i}^{2}\dim(W_{i})^{1/2}=c$ for
each $i = 1, \ldots, m$, so the proof follows from the previous
corollary.\end{proof}

\begin{rem}
By Corollary~\ref{C T 1 infty opt dual c 1}, the only
$(r,\infty)$-loss optimal component preserving dual fusion frames of
a uniform equi-dimensional Parseval fusion frame
$(\textbf{W},\mathbf{w})$ are the canonical ones
$(\textbf{W},\mathbf{v})$ with arbitrary vector of weights
$\textbf{v}$.
\end{rem}

In Example~\ref{Ej MasseyRuizStojanoff 1}, we are going to see a
fusion frame that has a unique (up to weights) loss optimal
$Q$-component preserving dual fusion frame with the same subspaces
as the canonical dual, but with $Q \neq
Q_{S_{\textbf{W},\mathbf{w}}^{-1}T_{\textbf{W},\mathbf{w}},\mathbf{v}}$
and therefore it gives another reconstruction formula.

\subsection{Optimal dual fusion frame systems for erasures of local frame
vectors}\label{S ODFFE vectors}

We will analyze now the situation where some local frame vectors are
lost. In this case we consider $J_{i} \subseteq L_{i},$ for each $i
= 1, \ldots, m$, $\mathcal{J}=(J_{1}, \ldots, J_{m})$ and
$|\mathcal{J}|=\sum_{i=1}^{m}|J_{i}|$. Let $M_{\mathcal{J}} \in
L(\oplus_{i=1}^{m}\mathbb{F}^{|L_{i}|},
\oplus_{i=1}^{m}\mathbb{F}^{|L_{i}|})$ be the self-adjoint operator
given by $M_{\mathcal{J}}((x_i^l)_{l \in
L_{i}})_{i=1}^{m}=((\chi_{J_{i}}(l)x_i^l)_{l \in L_{i}})_{i=1}^{m}$.

Fix $r \in \{1, \ldots, \sum_{i=1}^{m}|L_{i}|\}$. Let
$\mathcal{P}^{\times_{i=1}^{m}L_{i}}_{r}=\{\mathcal{J} :
|\mathcal{J}| = r\}$. By similar considerations to those in
section~\ref{S ODFFE subspaces}, we consider the vector error

\centerline{$e(r,(\mathbf{W},\mathbf{w},\mathcal{F}),(\mathbf{V},\mathbf{v},\mathcal{G}))=(\|T_{\mathbf{V},\mathbf{v}}C_{\mathcal{G}}M_{\mathcal{J}}C_{\mathcal{F}}^{*}T_{\textbf{W},\mathbf{w}}^{*}\|_{F})_{\mathcal{J}
\in \mathcal{P}^{\times_{i=1}^{m}L_{i}}_{r}}$,}

\noindent and define inductively

\centerline{$e_{1}^{(p)}(\textbf{W},\mathbf{w},\mathcal{F})=\text{inf}\{\|e(1,(\mathbf{W},\mathbf{w},\mathcal{F}),(\mathbf{V},\mathbf{v},\mathcal{G}))\|_{p}:
(\mathbf{V},\mathbf{v},\mathcal{G}) ~\text{is a dual fusion frame
system of}~ (\mathbf{W},\mathbf{w},\mathcal{F})\},$}

$\mathcal{D}_{1}^{(p)}(\textbf{W},\mathbf{w},\mathcal{F})$ as the
set of dual fusion frame systems
$(\mathbf{V},\mathbf{v},\mathcal{G})$ of
$(\textbf{W},\mathbf{w},\mathcal{F})$ with

\centerline{$\|e(1,(\textbf{W},\mathbf{w},\mathcal{F}),(\mathbf{V},\mathbf{v},\mathcal{G}))\|_{p}=e_{1}^{(p)}(\textbf{W},
\mathbf{w},\mathcal{F}),$}

\centerline{$e_{r}^{(p)}(\textbf{W},\mathbf{w},\mathcal{F})=\text{inf}\{
\|e(r,(\textbf{W},\mathbf{w},\mathcal{F}),(\mathbf{V},\mathbf{v},\mathcal{G}))\|_{p}:(\mathbf{V},\mathbf{v},\mathcal{G})
\in \mathcal{D}_{r-1}^{(p)}(\textbf{W},\mathbf{w},\mathcal{F})\},$}

\centerline{$\mathcal{D}_{r}^{(p)}(\textbf{W},\mathbf{w},\mathcal{F})=\{(\mathbf{V},\mathbf{v},\mathcal{G})
\in
\mathcal{D}_{r-1}^{(p)}(\textbf{W},\mathbf{w},\mathcal{F}):\|e(r,(\textbf{W},\mathbf{w},\mathcal{F}),(\mathbf{V},\mathbf{v},\mathcal{G}))\|_{p}=e_{r}^{(p)}(\textbf{W},\mathbf{w},\mathcal{F})\},$}

\noindent in case each
$\mathcal{D}_{r}^{(p)}(\textbf{W},\mathbf{w},\mathcal{F})$, called
the set of $(r,p)$-loss optimal dual fusion frames for
$(\textbf{W},\mathbf{w},\mathcal{F})$, is non empty.

In the following we consider the cases $p=2$ and $p=\infty$
obtaining results that are analogous to the ones viewed in
Section~\ref{S ODFFE subspaces}. For this, let $\mathcal{J} \in
\mathcal{P}^{\times_{i=1}^{m}L_{i}}_{1}$ where $J_{l}=\emptyset$ if
$l \neq i$ and $J_{i}=\{j\}$. Then
$M_{\mathcal{J}}C_{\mathcal{F}}^{*}T_{\textbf{W},\mathbf{w}}^{*}T_{\textbf{W},\mathbf{w}}C_{\mathcal{F}}M_{\mathcal{J}}^{*}=w_{i}^{2}\|f_i^j\|^{2}M_{\mathcal{J}}$,
and if $A \in
L(\bigoplus_{i=1}^{m}\mathbb{F}^{|L_{i}|},\mathcal{H})$,
\begin{equation}\label{E T mse local 1}
\|AM_{\mathcal{J}}C_{\mathcal{F}}^{*}T_{\textbf{W},\mathbf{w}}^{*}\|_{F}^{2}
=w_{i}^{2}\|f_i^j\|^{2}\|AM_{\mathcal{J}}\|_{F}^{2}.
\end{equation}
\subsubsection{The mean square error}\label{S ODFFE vectors p = 2}

Considering $p=2$ we obtain the mean square error,

\centerline{$\|e(r,(\mathbf{W},\mathbf{w},\mathcal{F}),(\mathbf{V},\mathbf{v},\mathcal{G}))\|_{2}=(\sum_{\mathcal{J}
\in
\mathcal{P}^{\times_{i=1}^{m}L_{i}}_{r}}\|T_{\mathbf{V},\mathbf{v}}C_{\mathcal{G}}M_{\mathcal{J}}C_{\mathcal{F}}^{*}T_{\textbf{W},\mathbf{w}}^{*}\|_{F}^{2})^{1/2}.$}

The following theorem about optimal $(r,2)$-loss optimal dual fusion
frame systems is similar to Theorem~\ref{T ODFF p=2}.

\begin{thm}\label{T ODFFE vector p=2}
Let $(\textbf{W},\mathbf{w},\mathcal{F})$ be a fusion frame system
for $\mathcal{H}$ where each element in $\mathcal{F}$ has norm equal
to $1$. Let
$\mathcal{G}_{c}=\{\frac{1}{w_{i}v_{i}}S_{\mathcal{F}}^{-1}\mathcal{F}_{i}\}_{i=1}^{m}$
and
$\mathcal{G}_{c,i}=\frac{1}{w_{i}v_{i}}S_{\mathcal{F}}^{-1}\mathcal{F}_{i}$.
Then
\begin{enumerate}
  \item $(S_{\mathcal{F}}^{-1}\textbf{W},\mathbf{v},\mathcal{G}_{c})$ is a $(r,2)$-loss optimal
component preserving dual fusion frame system for
$(\textbf{W},\mathbf{w},\mathcal{F})$.
  \item If $(\mathbf{V},\mathbf{v},\mathcal{G})$ is
a $(r,2)$-loss optimal dual fusion frame system of
$(\textbf{W},\mathbf{w},\mathcal{F})$, then
$T_{\mathbf{V},\mathbf{v}}C_{\mathcal{G}}=T_{S_{\mathcal{F}}^{-1}\textbf{W},\mathbf{v}}C_{\mathcal{G}_{c}}$.
  \item If
$(\mathbf{V},\mathbf{v},\mathcal{G})$ is a $(r,2)$-loss optimal
component preserving dual fusion frame system of
$(\textbf{W},\mathbf{w},\mathcal{F})$ then
$\mathbf{V}=S_{\mathcal{F}}^{-1}\textbf{W}$ and
$\mathcal{G}=\mathcal{G}_{c}$.
\end{enumerate}
\end{thm}

\begin{proof} By Remark~\ref{R wF marco sii WwF fusion frame system}, $\mathbf{w}\mathcal{F}$ is
a frame for $\mathcal{H}$, so $\mathcal{F}$ is also a frame for
$\mathcal{H}$ and $S_{\mathcal{F}}$ is a selfadjoint positive
invertible operator.

Note that
$T_{\mathcal{G}_{c,i}}=\frac{1}{w_{i}v_{i}}S_{\mathcal{F}}^{-1}T_{\mathcal{F}_{i}}$,
so $R(T_{\mathcal{G}_{c,i}}^{*}) \subseteq
R(T_{\mathcal{F}_{i}}^{*})$ and, by Remark~\ref{R Q Mi sum Wj subset
Mi sum Vj}, $C_{\mathcal{G}_{c}}C_{\mathcal{F}}^{*}$ is component
preserving.

We also have
$T_{S_{\mathcal{F}}^{-1}\textbf{W},\mathbf{v}}C_{\mathcal{G}_{c}}C_{\mathcal{F}}^{*}T_{\textbf{W},\mathbf{w}}^{*}=\sum_{i=1}^{m}S_{\mathcal{F}}^{-1}T_{\mathcal{F}_{i}}T_{\mathcal{F}_{i}}^{*}=S_{\mathcal{F}}^{-1}S_{\mathcal{F}}=I_{\mathcal{H}},$
thus $(S_{\mathcal{F}}^{-1}\textbf{W},\mathbf{v}\mathcal{G}_{c})$ is
a component preserving dual fusion frame system of
$(\textbf{W},\mathbf{w},\mathcal{F})$.

Using (\ref{E T mse local 1}), the rest of the proof is similar to
that of Theorem~\ref{T ODFF p=2}.\end{proof}

\subsubsection{The worst case error}\label{S ODFFE vectors p = infty}

For $p=\infty$ we have the worst-case error,

\centerline{$\|e(r,(\textbf{W},\mathbf{w},
\mathcal{F}),(\mathbf{V},\mathbf{v},
\mathcal{G}))\|_{\infty}=\max_{\mathcal{J} \in
\mathcal{P}^{\times_{i=1}^{m}L_{i}}_{r}}\|T_{\mathbf{V},\mathbf{v}}C_{\mathcal{G}}M_{\mathcal{J}}C_{\mathcal{F}}^{*}T_{\textbf{W},\mathbf{w}}^{*}\|_{F}.$}

The following results are analogous to Proposition~\ref{P W w set
duals}, Theorem~\ref{T ODFFE p infty canonico} and Corollary~\ref{C
T 1 infty opt dual c 1}, respectively. Their proofs follow similar
lines, so we omit them.

\begin{prop}\label{P W w set duals vector}
Let $(\textbf{W},\mathbf{w},\mathcal{F})$ be a fusion frame system
for $\mathcal{H}$ where each element in $\mathcal{F}$ is no null.
Then \begin{equation}\label{E P W w set duals vector}\{A \in
\mathfrak{L}_{C_{\mathcal{F}}^{*}T_{\textbf{W},\mathbf{w}}^{*}}:
\max_{\mathcal{J} \in
\mathcal{P}^{\times_{i=1}^{m}L_{i}}_{1}}\|AM_{\mathcal{J}}C_{\mathcal{F}}^{*}T_{\textbf{W},\mathbf{w}}^{*}\|_{F}
= \min_{B \in
\mathfrak{L}_{T_{\textbf{W},\mathbf{w}}^{*}}}\max_{\mathcal{J} \in
\mathcal{P}^{\times_{i=1}^{m}L_{i}}_{1}}\|BM_{\mathcal{J}}C_{\mathcal{F}}^{*}T_{\textbf{W},\mathbf{w}}^{*}\|_{F}\}\end{equation}is
non empty, compact and convex.
\end{prop}

\begin{thm}\label{T ODFFE vector p=infty}
Let $(\textbf{W},\mathbf{w},\mathcal{F})$ be a fusion frame system
for $\mathcal{H}$. If
$w_{i}\|f_{i}^j\|~\|S_{\mathbf{w}\mathcal{F}}^{-1}\pi_{\text{span}\{f_{i}^j\}}\|_{F}=c$
for each $i = 1, \ldots, m$, then
\begin{enumerate}
  \item
  $T_{S_{\mathbf{w}\mathcal{F}}^{-1}W,\mathbf{w}}C_{S_{\mathbf{w}\mathcal{F}}^{-1}\mathcal{F}}$
  is the unique element of the set (\ref{E P W w set duals vector}).
  \item $(S_{\mathbf{w}\mathcal{F}}^{-1}\textbf{W},\mathbf{w},S_{\mathbf{w}\mathcal{F}}^{-1}\mathcal{F})$ is the unique $(r,\infty)$-loss optimal component preserving dual fusion frame system of $(\textbf{W},\mathbf{w},\mathcal{F}).$
\end{enumerate}
\end{thm}

\begin{cor}\label{C 2 T ODFFE vector p=infty} Let $(\textbf{W},\mathbf{w},\mathcal{F})$ be a
fusion frame for $\mathcal{H}$. If $\mathbf{w}\mathcal{F}$ is
Parseval and $w_{i}\|f_i^j\|=c$ for each $i = 1, \ldots, m$, then
\begin{enumerate}
  \item $T_{\textbf{W},\mathbf{w}}C_{\mathcal{F}}$
  is the unique element of the set (\ref{E P W w set duals vector}).
  \item $(\textbf{W},\mathbf{v},\mathcal{F})$ is the unique $(r,\infty)$-loss optimal component preserving dual fusion frame system of $(\textbf{W},\mathbf{w},\mathcal{F})$.
\end{enumerate}
\end{cor}

A fusion frame system that has a unique $(r,2)$-loss optimal
$Q$-component preserving dual fusion frame system with the same
subspaces as the canonical dual, but with $Q \neq
Q_{S_{\textbf{W},\mathbf{w}}^{-1}T_{\textbf{W},\mathbf{w}},\mathbf{v}}$,
will be given in Example~\ref{Ej MasseyRuizStojanoff 1}. In this
case the optimal dual provides another reconstruction formula than
the canonical dual.

In \cite{Lopez-Han (2010)} and \cite{Leng-Han (2011)}, the spectral
norm is used instead of the Frobenius norm in the definition of the
worst-case error. We prefer the Frobenius norm in accordance with
the study made in subsection 5.1. Both worst-case errors coincide
for $r=1.$ So, by the used hierarchical definition and the relation
between dual fusion frame systems and dual frames, provided by
Theorem~\ref{T dual fusion frame systems}, we conclude that we can
obtain examples for Theorem~\ref{C 1 T ODFFE vector p=infty} and
Corollary~\ref{C 2 T ODFFE vector p=infty} from the examples for the
corresponding results in \cite{Lopez-Han (2010)} and \cite{Leng-Han
(2011)} (see Example~\ref{Leng-Han 3.5}).

\section{Examples}\label{SEjemplos}

\begin{example}
In \cite{Gavruta (2007)} a Bessel fusion sequence  $({\bf V},{\bf
v})$  is called an alternate dual of the fusion frame
$(\textbf{W},{\bf w})$ if for all $f \in \mathcal{H}$
\begin{equation}\label{E formula reconstruccion Gavruta}
f= \sum_{i=1}^{m}w_{i}v_{i}\pi_{V_{i}}S_{\textbf{W},{\bf
w}}^{-1}\pi_{W_i}(f).
\end{equation}
If $A: \mathcal{W} \rightarrow \mathcal{H}$, $A(f_i)_{i=1}^{m}=
\sum_{i=1}^{m}v_{i}\pi_{V_{i}}S_{\textbf{W},{\bf w}}^{-1}f_i$, then
by (\ref{E formula reconstruccion Gavruta}) $A \in
\mathfrak{L}_{T_{\textbf{W},{\bf w}}^{*}}$. We have
$\widetilde{V}_{i}:=AM_i\mathcal{W}=\pi_{V_i} S_{\textbf{W},{\bf
w}}^{-1} W_i$ and $Q_{A,\mathbf{v}} : \mathcal{W} \rightarrow
\widetilde{\mathcal{V}}$, $Q_{A,\mathbf{v}}(f_i)_{i=1}^m=(\pi_{V_i}
S_{\textbf{W},{\bf w}}^{-1} f_{i})_{i=1}^{m}.$ By Theorem~\ref{T V,v
dual fusion frame sii Vi=ApiWj}, $(\widetilde{\mathbf{V}},
\mathbf{v})$ is a $Q_{A,\mathbf{v}}$-component preserving dual
fusion frame of $(\textbf{W}, \mathbf{w})$. By Lemma~\ref{L QA comp
preserv}, (\ref{E formula reconstruccion Gavruta}) can be written
using this dual fusion frame as $f=T_{\widetilde{\mathbf{V}},{\bf
v}}Q_{A,\mathbf{v}}T_{\textbf{W},{\bf w}}^{*}f$.
\end{example}

\begin{example}\label{Ej MasseyRuizStojanoff 2} Let
$\mathcal{H}=\mathbb{C}^{4}$, $w_{1} > 0$, $w_{2} > 0$,
$W_{1}=\{(x_{1}, x_{2}, 0 , 0) : x_{1}, x_{2} \in \mathbb{C}\}$ and
$W_{2}=\{(0, x_{2}, x_{3} , -x_{2}) : x_{2}, x_{3} \in
\mathbb{C}\}$. Then $(\mathbf{W},\mathbf{w})$ is a
$2$-equi-dimensional Riesz fusion basis for $\mathbb{C}^{4}$ and so
its unique component preserving duals are the canonical ones. Set
$w_{1}=w_{2}=1$.

{\bf (a)} Although $(\mathbf{W},1)$ is a Riesz fusion basis, it is
possible to construct a dual fusion frame which is different from
the canonical ones. For this, let

\centerline{$\mathcal{F}_{1}=\{(1,0,0,0),(0,1,0,0), (1,0,0,0)\},~~
\mathcal{F}_{2}=\{(0,1,0,-1),(0,0,1,0),(0,0,1,0)\},$}

\centerline{$\mathcal{G}_{1}=\{(\frac{1}{2},\frac{1}{2},-\frac{1}{2},0),(0,1,0,1),(\frac{1}{2},-\frac{1}{2},\frac{1}{2},0)\},~~
\mathcal{G}_{2}=\{(0,0,0,-1),(\frac{1}{2},-\frac{1}{2},\frac{1}{2},0),(-\frac{1}{2},\frac{1}{2},\frac{1}{2},0)\}.$}

\noindent Then $\{\mathbf{W},\mathbf{w},\mathcal{F}\}$ is a fusion
frame system for $\mathbb{C}^{4}$ and $\mathcal{G}$ is a dual frame
of $\mathcal{F}$ that is not the canonical one.

Let $V_i=\text{span}\mathcal{G}_{i},\,i=1,2.$ By Theorem~\ref{T dual
fusion frame systems}, $(\mathbf{V},1,\mathcal{G})$ is a dual fusion
frame system of $(\mathbf{W},1,\mathcal{F})$.

Note that $C_{\mathcal{G}}C_{\mathcal{F}}^*: \mathcal{W} \rightarrow
\mathcal{V}, C_{\mathcal{G}}C_{\mathcal{F}}^*((x_{1}, x_{2}, 0,
0),(0, y_{2}, y_{3}, -y_{2}))=((x_{1}, x_{2}, 0, x_2),(0, 0, y_{3},
-2y_{2}))$ is block-diagonal but not component preserving.

Since $\dim(V_i)=3>\dim(W_i)=2,\,i=1,2,$ $(\mathbf{V},1)$ gives a
dual fusion frame which is different from the canonical one,
moreover, it is not a Riesz fusion basis.

{\bf (b)} Consider $T_{1}: \mathbb{C}^{2} \rightarrow
\mathbb{C}^{4}$, $T_{1}(x_{1}, x_{2})=(x_{1}, x_{2}, 0, 0)$ and
$T_{2}: \mathbb{C}^{2} \rightarrow \mathbb{C}^{4}$, $T_{2}(x_{1},
x_{2})=(0, \frac{1}{\sqrt{2}}x_{2}, x_{1},
-\frac{1}{\sqrt{2}}x_{2})$. Then $(T_{1}, T_{2})$ is a projective
Riesz $(2,2,\mathbb{C}^{4})$-reconstruction system associated with
$(\textbf{W},{\bf w})$ with a unique dual, the canonical one, that
is not projective (see \cite{Massey-Ruiz-Stojanoff (2012b)}, Example
5.4).
\end{example}

\begin{example}\label{Ej MasseyRuizStojanoff 1}
Let $\mathcal{H}=\mathbb{F}^{3}$, $W_{1}=\{(1, 0, 0)\}^{\perp}$,
$W_{2}=\{(0, 1, 0)\}^{\perp}$, $w_{1} > 0$ and $w_{2}
> 0$. Then $(\mathbf{W},\mathbf{w})$ is a $2$-equi-dimensional fusion frame for $\mathbb{F}^{3}$
with $S_{\textbf{W},{\bf w}}^{-1}(x_{1}, x_{2},
x_{3})=(\frac{x_1}{w_{2}^{2}},\frac{x_2}{w_{1}^{2}},
\frac{x_3}{w_{1}^{2}+w_{2}^{2}}).$

{\bf (a)} Any $A \in \mathfrak{L}_{T_{\textbf{W},\mathbf{w}}^{*}}$
is given by
\begin{align*}
A&\paren{(0,x_{2},x_{3}),(y_{1}, 0, y_{3})}=\\\notag
&=\paren{r_{11}x_{3}+\frac{y_{1}}{w_{2}}+r_{12}y_{3},
\frac{x_{2}}{w_{1}}+r_{21}x_{3}+r_{22}y_{3},
(\frac{w_{1}}{w_{1}^{2}+w_{2}^{2}}+r_{31})x_{3}+(\frac{w_{2}}{w_{1}^{2}+w_{2}^{2}}+r_{32})y_{3}},
\end{align*}
where $r_{i1}w_{1}+r_{i2}w_{2}=0$ for $i=1,2,3$.

By Theorem~\ref{T V,v dual fusion frame sii Vi=ApiWj}, any
$Q_{A,\mathbf{v}}$-component preserving dual fusion frame has
arbitrary weights $v_{1}, v_{2}$ and subspaces

\centerline{$V_{1}=AM_{1}\mathcal{W}=\text{span}\{(0,1,0),(r_{11},r_{21},\frac{w_{1}}{w_{1}^{2}+w_{2}^{2}}+r_{31})\},$}
\centerline{$V_{2}=AM_{2}\mathcal{W}=\text{span}\{(1,0,0),(r_{12},r_{22},\frac{w_{2}}{w_{1}^{2}+w_{2}^{2}}+r_{32})\}.$}

We have
$Q_{S_{\textbf{W},\mathbf{w}}^{-1}T_{\textbf{W},\mathbf{w}},\mathbf{v}}((0,x_{2},x_{3}),(y_{1},
0, y_{3}))=(\frac{1}{v_{1}}(0, \frac{x_{2}}{w_{1}},
\frac{w_{1}x_{3}}{w_{1}^{2}+w_{2}^{2}}),\frac{1}{v_{2}}(\frac{y_{1}}{w_{2}},
0, \frac{w_{2}y_{3}}{w_{1}^{2}+w_{2}^{2}})).$ We note that
$\set{\paren{W_{1},w_{1}},
\paren{W_{2},w_{2}}}$ is not Parseval and has the same subspaces as its canonical
dual.

{\bf (b)} Let $Q: \mathcal{W} \rightarrow \mathcal{V},$
$Q((0,x_{2},x_{3}),(y_{1}, 0, y_{3}))=(\frac{1}{v_{1}}(0,
\frac{x_{2}}{w_{1}},
\frac{x_{3}}{2w_{1}}),\frac{1}{v_{2}}(\frac{y_{1}}{w_{2}}, 0,
\frac{y_{3}}{2w_{2}}))$.

If $S_{D}$ and $Q_D$ are as in Theorem~\ref{T ODFF p=2}, then
$S_{D}^{-1}\textbf{W}=\textbf{W}$ and $Q_{D}=Q$. By Theorem~\ref{T
ODFF p=2}(1), $(\textbf{W},\mathbf{v})$ is a $(r,2)$-loss optimal
$Q$-component preserving dual fusion frame of
$(\textbf{W},\mathbf{w})$.

The unique element in the set (\ref{E P W w set duals}) is given by
$A((0,x_{2},x_{3}),(y_{1}, 0,
y_{3}))=(\frac{y_{1}}{w_{2}},\frac{x_{2}}{w_{1}},\frac{1}{2}(\frac{x_{3}}{w_{1}}+\frac{y_{3}}{w_{2}}))$.
In this case, $V_{1}=W_{1}$, $V_{2}=W_{2}$ and $Q_{A,\mathbf{v}}=Q$.
Therefore, $(\textbf{W},\mathbf{v})$ is the unique (up to weights)
$(r,\infty)$-loss optimal $Q$-component preserving dual fusion frame
of $(\textbf{W},\mathbf{w}).$

{\bf (c)} Let
$\mathcal{F}_{1}=\{(0,0,1),(0,\frac{3}{2},-\frac{1}{2}),-(0,\frac{\sqrt{3}}{2},\frac{1}{2})\}$
and $
\mathcal{F}_{2}=\{(0,0,1),(\frac{\sqrt{3}}{2},0,-\frac{1}{2}),-(\frac{\sqrt{3}}{2},0,\frac{1}{2})\}.$
\noindent Then $\mathcal{F}_{i}$ is a unit norm $\frac{3}{2}$-tight
frame for $W_i$, $i=1,2,$ $(\mathbf{W},\mathbf{w},\mathcal{F})$ is a
fusion frame system for $\mathbb{F}^{4}$ and $\mathcal{F}$ is a
frame for $\mathbb{F}^{4}$ with
$S_{\mathcal{F}}(x_{1},x_{2},x_{3})=(\frac{3}{2}x_{1},\frac{3}{2}x_{2},3x_{3})$.

Let $\mathcal{G}_{c}=\{\{g_i^l\}_{l=1}^3\}_{i=1}^{2}$ be as in
Theorem~\ref{T ODFFE vector p=2}, i.e.

\centerline{$g_{1}^{1}=\frac{1}{w_{1}v_{1}}(0,0,\frac{1}{3}),~~g_{1}^{2}=\frac{1}{w_{1}v_{1}}(0,\frac{\sqrt{3}}{3},\frac{-1}{6}),~~g_{1}^{3}=\frac{1}{w_{1}v_{1}}(0,\frac{-\sqrt{3}}{3},\frac{-1}{6}),$}

\centerline{$g_{2}^{1}=\frac{1}{w_{2}v_{2}}(0,0,\frac{1}{3}),~~g_{2}^{2}=
\frac{1}{w_{2}v_{2}}(\frac{\sqrt{3}}{3},0,\frac{-1}{6}),~~g_{2}^{3}=\frac{1}{w_{2}v_{2}}(\frac{-\sqrt{3}}{3},0,\frac{-1}{6}).$}
By Theorem~\ref{T ODFFE vector p=2},
$(\textbf{W},\mathbf{v},\mathcal{G}_{c})$ is the unique (up to
weights) $(r,2)$-loss optimal component preserving dual fusion frame
system of $(\textbf{W},\mathbf{w},\mathcal{F})$ and
$C_{\mathcal{G}_{c}}C_{\mathcal{F}}^{*}=Q$.

If $w_{1} \neq w_{2}$, then $Q \neq
Q_{S_{\textbf{W},\mathbf{w}}^{-1}T_{\textbf{W},\mathbf{w}},\mathbf{v}}$
and $T_{\textbf{W},\mathbf{v}}Q \neq
T_{\textbf{W},\mathbf{v}}Q_{S_{\textbf{W},\mathbf{w}}^{-1}T_{\textbf{W},\mathbf{w}},\mathbf{v}}$.
So, by the analysis done in (b) and (c), $(\textbf{W},\mathbf{w})$
has loss optimal duals with the same subspaces as the canonical ones
but that provide different reconstruction formulas than the
canonical duals.

{\bf (d)} Let $T_{1}: \mathbb{F}^{2} \rightarrow \mathbb{F}^{3}$,
$T_{1}(x_{1}, x_{2})=(0,x_{1}, x_{2})$, and $T_{2}: \mathbb{F}^{2}
\rightarrow \mathbb{F}^{3}$, $T_{2}(x_{1}, x_{2})=(x_{1},0,x_{2})$.
Then $(T_{1},T_{2})$ is a projective
$(2,2,\mathbb{F}^{3})$-reconstruction system associated with
$(\mathbf{W},\mathbf{w})$. This reconstruction system is considered
in \cite[Example 5.1]{Massey-Ruiz-Stojanoff (2012b)}, where it is
shown that if $\mathbb{F}=\mathbb{C}$, $(T_{1},T_{2})$ has
projective duals but the canonical dual is not projective, and if
$\mathbb{F}=\mathbb{R}$, $(T_{1},T_{2})$ has not projective
duals.\end{example}

\begin{example}\label{Leng-Han 3.5}
Let $\mathcal{H}=\mathbb{F}^{3},$
$\mathcal{F}_1=\{(1,0,0),(0,1,0),(-2,1,1)\},$ $W_1=\text{span}\,
\mathcal{F}_1,$ $\mathcal{F}_2=\{(1,-2,-1)\},$ $W_2=\text{span}\,
\mathcal{F}_2,$  $\mathcal{G}_1=\{(\frac{22-\sqrt{74}}{20},
\frac{2-\sqrt{74}}{20},\frac{3}{2}),
(\frac{2-\sqrt{74}}{20},\frac{22-\sqrt{74}}{20},-\frac{3}{2}),(\frac{2-\sqrt{74}}{20},
\frac{2-\sqrt{74}}{20},\frac{1}{2})\},$ $V_1=\text{span}\,
\mathcal{G}_1$ $\mathcal{G}_2=\{
(\frac{2-\sqrt{74}}{20},\frac{2-\sqrt{74}}{20},-\frac{1}{2})\}$ and
$V_2=\text{span}\, \mathcal{G}_2.$

Then $(\textbf{W},1,\mathcal{F})$ is a fusion frame system for
$\mathcal{H}$, $V_1 = S_{\textbf{W},1}^{-1}W_1$ and $V_2\neq
S_{\textbf{W},1}^{-1}W_2.$ By \cite[Example~3.5]{Leng-Han (2011)},
$(\mathbf{V},1,\mathcal{G})$ is the $(r,\infty)$-loss optimal dual
fusion frame system of $(\textbf{W},1,\mathcal{F})$, and it is
different from the canonical ones.
\end{example}

\section*{Acknowledgement}
S. B. Heineken acknowledges the support of Grant UBACyT 2011-2014
(UBA). The research of P. M. Morillas has been partially supported
by Grant P-317902 (UNSL). Both thank the valuable comments and
suggestions of the referee that significantly improved the
presentation of the paper.



\end{document}